\documentclass[a4paper, reqno, 14pt]{amsart}

\usepackage[usenames,dvipsnames]{color}
\usepackage{amsthm,amsfonts,amssymb,amsmath,amsxtra}
\usepackage[all]{xy}
\SelectTips{cm}{}
\usepackage{xr-hyper}
\usepackage[colorlinks=
   citecolor=Black,
   linkcolor=Red,
   urlcolor=Blue]{hyperref}
\usepackage{verbatim}

\usepackage[margin=1.25in]{geometry}
\usepackage{mathrsfs}

\usepackage{diagrams}

\RequirePackage{xspace}
\RequirePackage{etoolbox}
\RequirePackage{varwidth}
\RequirePackage{enumitem}
\RequirePackage{tensor}
\RequirePackage{mathtools}
\RequirePackage{longtable}
\RequirePackage{multirow}

\def\ge{\geqslant}
\def\le{\leqslant}
\def\a{\alpha}
\def\b{\beta}
\def\g{\gamma}
\def\G{\Gamma}

\def\L{\Lambda}

\def\s{\sigma}
\def\t{\tau}

\def\k{\kappa}
\def\l{\lambda}

\def\i{^{-1}}

\def\<{\langle}
\def\>{\rangle}

\newcommand{\fka}{\ensuremath{\mathfrak{a}}\xspace}

\def\brF{\breve F}
\def\brW{\breve W}
\def\brI{\breve \CI}
\def\brK{\breve \CK}
\def\brS{\breve S}
\def\brT{\breve T}

\def\brbS{\breve{\mathbb S}}

\newcommand{\BF}{\ensuremath{\mathbb {F}}\xspace}
\newcommand{{\BG}}{\ensuremath{\mathbb {G}}\xspace}

\newcommand{{\BK}}{\ensuremath{\mathbb {K}}\xspace}

\newcommand{\BN}{\ensuremath{\mathbb {N}}\xspace}

\newcommand{\BQ}{\ensuremath{\mathbb {Q}}\xspace}

\newcommand{\BZ}{\ensuremath{\mathbb {Z}}\xspace}

\newcommand{\CI}{\ensuremath{\mathcal {I}}\xspace}

\newcommand{\CK}{\ensuremath{\mathcal {K}}\xspace}

\newcommand{\CO}{\ensuremath{\mathcal {O}}\xspace}

\newcommand{\CU}{\ensuremath{\mathcal {U}}\xspace}

\newcommand{\Ad}{{\mathrm{Ad}}}
\newcommand{\ad}{{\mathrm{ad}}}

\DeclareMathOperator{\Aut}{Aut}

\DeclareMathOperator{\Adm}{Adm}

\DeclareMathOperator{\Gal}{Gal}

\DeclareMathOperator{\Lie}{Lie}

\DeclareMathOperator{\Pic}{Pic}

\DeclareMathOperator{\Spec}{Spec}

\def\kk{\mathbf k}

\newtheorem{theorem}{Theorem}
\newtheorem{proposition}[theorem]{Proposition}
\newtheorem{lemma}[theorem]{Lemma}

\newtheorem{corollary}[theorem]{Corollary}

\theoremstyle{definition}
\newtheorem{definition}[theorem]{Definition}

\newtheorem{remark}[theorem]{Remark}

\numberwithin{equation}{section}
\numberwithin{theorem}{section}

\setitemize[0]{leftmargin=*,itemsep=\the\smallskipamount}
\setenumerate[0]{leftmargin=*,itemsep=\the\smallskipamount}

\renewcommand{\to}{%
   \ifbool{@display}{\longrightarrow}{\rightarrow}%
   }
\let\shortmapsto\mapsto
\renewcommand{\mapsto}{%
   \ifbool{@display}{\longmapsto}{\shortmapsto}%
   }
\newlength{\olen}
\newlength{\ulen}
\newlength{\xlen}
\newcommand{\xra}[2][]{%
   \ifbool{@display}%
      {\settowidth{\olen}{$\overset{#2}{\longrightarrow}$}%
       \settowidth{\ulen}{$\underset{#1}{\longrightarrow}$}%
       \settowidth{\xlen}{$\xrightarrow[#1]{#2}$}%
       \ifdimgreater{\olen}{\xlen}%
          {\underset{#1}{\overset{#2}{\longrightarrow}}}%
          {\ifdimgreater{\ulen}{\xlen}%
             {\underset{#1}{\overset{#2}{\longrightarrow}}}
             {\xrightarrow[#1]{#2}}}}%
      {\xrightarrow[#1]{#2}}
   }
\makeatother
\newcommand{\xyra}[2][]{%
   \settowidth{\xlen}{$\xrightarrow[#1]{#2}$}%
   \ifbool{@display}%
      {\settowidth{\olen}{$\overset{#2}{\longrightarrow}$}%
       \settowidth{\ulen}{$\underset{#1}{\longrightarrow}$}%
       \ifdimgreater{\olen}{\xlen}%
          {\mathrel{\xymatrix@M=.12ex@C=3.2ex{\ar[r]^-{#2}_-{#1} &}}}%
          {\ifdimgreater{\ulen}{\xlen}%
             {\mathrel{\xymatrix@M=.12ex@C=3.2ex{\ar[r]^-{#2}_-{#1} &}}}
             {\mathrel{\xymatrix@M=.12ex@C=\the\xlen{\ar[r]^-{#2}_-{#1} &}}}}}%
      {\mathrel{\xymatrix@M=.12ex@C=\the\xlen{\ar[r]^-{#2}_-{#1} &}}}%
   }
\makeatletter
\newcommand{\xla}[2][]{%
   \ifbool{@display}%
      {\settowidth{\olen}{$\overset{#2}{\longleftarrow}$}%
       \settowidth{\ulen}{$\underset{#1}{\longleftarrow}$}%
       \settowidth{\xlen}{$\xleftarrow[#1]{#2}$}%
       \ifdimgreater{\olen}{\xlen}%
          {\underset{#1}{\overset{#2}{\longleftarrow}}}%
          {\ifdimgreater{\ulen}{\xlen}%
             {\underset{#1}{\overset{#2}{\longleftarrow}}}
             {\xleftarrow[#1]{#2}}}}%
      {\xleftarrow[#1]{#2}}
   }
\newcommand{\isoarrow}{%
   \ifbool{@display}{\overset{\sim}{\longrightarrow}}{\xrightarrow\sim}%
   }
   
\begin{document}

\title[On the connected components of affine Deligne-Lusztig varieties]{On the connected components of affine Deligne-Lusztig varieties}
\author[X. He]{Xuhua He}
\address{Department of Mathematics, University of Maryland, College Park, MD 20742 and Institute for Advanced Study, Princeton, NJ 08540}
\email{xuhuahe@math.umd.edu}
\author[R. Zhou]{Rong Zhou}
\address{Department of Mathematics, Harvard University, Cambridge, MA 02138}
\email{rzhou@math.harvard.edu}
\thanks{X. H. was partially supported by NSF DMS-1463852.}


\date{\today}

\begin{abstract}
We study the set of connected components of certain unions of affine Deligne-Lusztig varieties arising from the study of Shimura varieties. We determine the set of connected components for basic $\s$-conjugacy classes. As an application, we verify the Axioms in \cite{HR} for certain PEL type Shimura varieties. We also show that for any nonbasic $\s$-conjugacy class in a residually split group,  the set of connected components is ``controlled'' by the set of straight elements associated to the $\s$-conjugacy class, together with the obstruction from the corresponding Levi subgroup. Combined with \cite{Zhou}, this allows one to verify in the residually split case, the description of the mod-$p$ isogeny classes on Shimura varieties conjectured by Langland and Rapoport. Along the way, we determine the Picard group of the Witt vector affine Grassmannian of \cite{BS} and \cite{Zhu} which is of independent interest.
\end{abstract}

\maketitle


\section*{Introduction}

\subsection{} Let $F$ be a non-archimedean local field with valuation ring $\CO_F$ and residue field $\BF_q$ and $\brF$ be the completion of the maximal unramified extension of $F$. Let $G$ be a connected reductive group over $F$ and let $\s$ be the Frobenius automorphism on $G(\brF)$. Let $\{\mu\}$ be a geometric conjugacy class of cocharacters of $G$. Let $b \in G(\breve F)$ and $\breve K$ be a $\s$-invariant parahoric subgroup of $G(\breve F)$. The (union of) affine Deligne-Lusztig varieties associated to $(G, \{\mu\}, b, \breve K)$ is defined to be $$X(\{\mu\}, b)_K=\{g \in G(\breve F)/\breve K; g \i b \s(g) \in \breve K \Adm(\{\mu\})_K \breve K\},$$ where $\Adm(\{\mu\})_K$ is the admissible set associated to $\{\mu\}$. 

The motivation to study $X(\{\mu\}, b)_K$ comes from the theory of Shimura varieties. If $(G, X)$ is a Shimura datum with $\{\mu\}$ the conjugacy class of the inverse of the Hodge cocharacter and $F=\BQ_p$, the set $X(\{\mu\}, b)_K$ is closely related to the $\overline{\mathbb{F}}_p$-points of the corresponding Shimura variety with parahoric level structure $\breve K$, and in special cases to the $\overline{\mathbb{F}}_p$ points of a moduli space of $p$-divisible groups defined by Rapoport and Zink \cite{RZ}. 

In the equal characteristic case, $X(\{\mu\}, b)_K$ is the set of $\overline{\mathbb{F}}_q$-valued points of a locally closed, locally of finite type subscheme of the partial affine flag variety $G(\breve F)/\breve K$. Thanks to the recent work of Zhu \cite{Zhu} and Bhatt-Scholze \cite{BS}, we may, even in the mixed characteristic case, endow $X(\{\mu\}, b)_K$ with the structure of a perfect scheme over $\BF_q$. More precisely, it is a locally closed subscheme of the Witt vector partial affine flag variety $Gr_{\mathcal{G}}$ (here $\mathcal{G}$ is the group scheme associated to $\brK$). Therefore topological notions such as irreducible and connected components make sense.

If the group $G$ is unramified, $\mu$ is minuscule and $\breve K$ is a hyperspecial parahoric subgroup, $\breve K \Adm(\{\mu\})_K \breve K$ is a single double coset of $\breve K$.  The set of connected components of $X(\{\mu\}, b)_K$ for split group $G$, hyperspecial parahoric $\CK$ and arbitrary cocharacter $\mu$ is determined by Viehmann \cite{Vi}. For other unramified groups with hyperspecial parahoric $\CK$, Chen, Kisin and Viehmann \cite{CKV} determine the set of connected components  under the assumption that $\mu$ is minuscule and Nie \cite{Nie} extends the result to non-minuscule cocharacters. 

\subsection{}The aim of this paper is to study the set of connected components of $X(\{\mu\}, b)_K$ for any reductive group $G$ and any parahoric subgroup $\breve K$. 

Let us first recall the nonemptiness pattern of $X(\{\mu\}, b)_K$. Let $B(G, \{\mu\})$ be the set of neutral acceptable elements, a certain subset of the $\s$-conjugacy classes of $G(\breve F)$ (see \S\ref{2} for the precise definitions). Then it is known that $X(\{\mu\}, b)_K$ is nonempty if and only if the $\s$-conjugacy class of $b$ lies in $B(G, \{\mu\})$. This was conjectured by Kottwitz and Rapoport in \cite{KR} and proved for quasi-split groups by Wintenberger \cite{Wi} and in the general case by the first author in \cite{He00}. 

Note that the set of connected components of the partial affine flag $G(\breve F)/\breve K$ is isomorphic to $\pi_1(G)_{\G_0}$, where $\pi_1(G)_{\G_0}$ is the set of coinvariants of the fundamental group of $G$ under the actions of the Galois group $\G_0=\Gal(\bar F/\breve F)$. This gives the first obstruction to connecting points of $X(\{\mu\}, b)_K \subset G(\breve F)/\breve K$. Our first main result is that for basic $b$, this is the only obstruction except in trivial  cases. See Theorem \ref{conn}. 

\begin{theorem}\label{0.1}
If $\mu$ is essentially noncentral, then for basic $b \in B(G, \{\mu\})$, $$\pi_0(X(\{\mu\}, b)_K) \cong \pi_1(G)_{\G_0}^\s.$$
\end{theorem}

\subsection{} Now we discuss some application. In  joint work \cite{HR} of the first author with Rapoport, there is the formulation of five axioms on  Shimura varieties with parahoric level structure. These axioms allow us to apply group-theoretic techniques to study many questions on certain characteristic subsets (Newton stratification, the Ekedahl-Oort stratification and the Kottwitz-Rapoport stratification, etc.) in the reduction modulo p of a general Shimura variety with parahoric level structure. It is shown in \cite{HR} that if the axioms are satisfied, then the Newton strata are nonempty in their natural range, and that the Grothendieck conjecture on the closure relations of Newton strata is true. We refer to \cite{HR} for more details and for other consequences of the axioms. These axioms are also used in an essential way in the work \cite{GHN2} on a family of Shimura varieties with nice geometric properties on both the basic locus and on all the non-basic Newton strata and the forthcoming work \cite{HN2} on the density problem of the $\mu$-ordinary locus of Shimura varieties. 

One major difficulty in verifying the axioms is to show that for Shimura varieties with Iwahori level structure, the minimal Kottwitz-Rapoport stratum intersects every connected component of it. As an application of Theorem \ref{0.1}, we show in Proposition \ref{10.1} that for certain PEL type Shimura varieties (unramified of type A and C and odd ramified unitary groups), every point in the basic locus is connected to a point in the minimal KR stratum. In \S\ref{10}, we verify the axioms on these Shimura varieties. In particular, combined with \cite{HR}, we finish the proof of the nonemptiness pattern and closure relations on the Newton strata of these Shimura varieties. 

\subsection{} Now we come to the second main result.
Let $\breve W$ be the Iwahori-Weyl group of $G(\breve F)$. The Frobenius morphism $\s$ on $G(\breve F)$ induces a group automorphism $\s$ on $\breve W$. The embedding $\breve W \to G(\breve F)$ induces a map from the $\s$-conjugacy classes of $\breve W$ to the set of $\s$-conjugacy classes of $G(\breve F)$. This map is not injective. However, it is proved by the first author that there is a natural bijection between the set of straight $\s$-conjugacy classes of $\breve W$ and the set of $\s$-conjugacy classes on $G(\breve F)$ (see \cite{He99}), and for any $[b] \in B(G, \{\mu\})$, there exists a straight representative in the admissible set of $\mu$ (see \cite{He00}). 

For basic $[b]$, there is only one straight element in the associated straight $\s$-conjugacy class in $\breve W$ which lies in $\Adm(\{\mu\})$. For nonbasic $[b]$, the straight elements of $\Adm(\{\mu\})$ in the associated straight $\s$-conjugacy class in $\breve W$ is not unique anymore. However, if $G$ is residually split, then for $[b] \in B(G, \{\mu\})$, all the straight representatives of $[b]$ lie in the admissible set of $\mu$. 

As mentioned above, the set of connected components of $G(\breve F)/\breve K$ give the first obstruction to the connectedness of $X(\{\mu\}, b)_K$. For basic $b$, as we have proved in Theorem \ref{0.1}, this is the only obstruction in general. For nonbasic $[b] \in B(G, \{\mu\})$, there is another obstruction, coming from the non-uniqueness of the straight representatives of $[b]$. If there are more than one straight representatives of $[b]$, then one cannot expect in general that the points associated to different straight elements are connected. For example in the case of split $G$ and minuscule $\mu$, when $b$ is a translation element, the only non-empty strata in $X(\{\mu\},b)$ are the ones corresponding to translation elements. In particular these are all open in $X(\{\mu\},b)$ and one cannot connect points in different strata.

Our second main result is that in the residually split case, these are all the obstructions to the connectedness of $X(\{\mu\}, b)_K$.

\begin{theorem}\label{0.2}
Let $G$ be residually split and $[b] \in B(G, \{\mu\})$. Assume that $[b]$ is essentially nontrivial in the associated Levi subgroup. Then we have a surjection $$\coprod_{w \in \brW \text{ is a straight element with } \dot w \in [b]} \pi_1(M_{v_w})_{\Gamma_0} \twoheadrightarrow \pi_0(X(\{\mu\}, b))_K.$$ 
\end{theorem}
We refer to \S\ref{8} to the notations in the theorem and the more general cases where the ``essentially nontrivial'' assumption is dropped. 

In fact, for an unramified group and its hyperspecial parahoric subgroup, it is shown in \cite{Vi}, \cite{CKV} and \cite{Nie} that the precise description of the connected components for nonbasic elements is given by the Hodge-Newton decomposability. It is a challenging problem to determine explicitly the connected components for general $G$ and $\CK$. However, the above result is enough for some important application, which we discuss in the next subsection. 

\subsection{} In the seminal work \cite{Ki}, Kisin showed that the mod-$p$ points on a Shimura variety of abelian type with hyperspecial level structure agrees with the conjectural description given in the Langlands-Rapoport conjecture.\footnote{More precisely the description of the isogeny classes differ up to conjugation by a certain group element}. The key part of the proof is to define a certain map from $X(\{\mu\},b)_K$ into a mod-$p$ isogeny class. His strategy consists of a local and a global part: first use a deformation theoretic argument to show that a map from the affine Deligne-Lusztig variety to the Shimura variety is well-defined on a connected component once it is defined on a point; then show that every connected component of the affine Deligne-Lusztig variety contains a point where the map is well-defined using isogenies which lift to characteristic $0$. The last part uses in a crucial way the explicit description of the set of connected components of $X(\{\mu\},b)_K$ determined in \cite{CKV}. 

For hyperspecial level structure, only a single affine Deligne-Lusztig variety occurs; for other parahoric level structure, one has to consider the union of affine Deligne-Lusztig varieties. This is one of the major new difficulties in the study of arbitrary parahoric level structure. In \cite{Zhou}, the second author generalizes the deformation theoretic argument to the integral models of Shimura varieties of Hodge type  with arbitrary parahoric level structure constructed in \cite{KP} (these models are constructed under a tameness hypothesis on the group at $p$). One innovative result is that one may move between different Levi subgroups using isogenies which lift to characteristic $0$. In other words, for the application to the Langlands-Rapoport conjecture for residually split groups, one does not need to know whether any two straight elements in Theorem \ref{0.2} are connected. Thus our Theorem \ref{0.2} provides enough local information for this purpose. Combining the global argument \cite{Zhou} together with the local result here, one deduces that the mod-$p$ isogeny classes have the form predicted in \cite{LR} for those Shimura varieties associated to groups which are residually split and have arbitrary parahoric level at $p$.

\subsection{} Our proof of Theorem \ref{0.1} and Theorem \ref{0.2} is different from the proof of  \cite{Vi}, \cite{CKV} and \cite{Nie} in the case of hyperspecial level structure. At the time, the scheme structure on the mixed characteristic affine Deligne Lusztig varieties was not known and the authors worked with an ad hoc definition of connected components. A consequence of this was that their proofs were completely combinatorial. By contrast, we work in the Zariski topology, and our proofs use both geometry and combinatorics. 

More precisely, we rely on the following key ingredients:

\begin{itemize}
\item The relation between the straight elements in the admissible set $\Adm(\{\mu\})$ and the $\s$-conjugacy classes in the neutral acceptable set $B(G, \{\mu\})$ established in \cite{He00} and the reduction method introduced in \cite{He99};

\item The line bundles on the affine flag varieties and the quasi-affineness of irreducible components of affine Deligne-Lusztig varieties;

\item The structure of the $\s$-centralizer $J_b$ of $b$ and the construction of explicit curves in $X(\{\mu\}, b)$ for each affine root subgroup of $J_b$;

\end{itemize}

\subsection{} Now let us provide more details for the Iwahori case. We simply write $X(\{\mu\}, b)$ for $X(\{\mu\}, b)_I$, where $\breve I$ is the Iwahori subgroup. 

Note that $X(\{\mu\}, b)=\sqcup_{w \in \Adm(\{\mu\})} X_w(b)$ is a union of affine Deligne-Lusztig varieties, where $X_w(b)=\{g \in G(\breve F)/\breve I; g \i b \s(g) \in \brI \dot w \brI\}$. In general, the structure of $X_w(b)$ is very complicated. However, it is proved in \cite{He99} that if $w$ is $\s$-straight, then $X_w(b)$ is either empty or discrete with transitive action of $J_b$. 

Our first major step is to show that every point in $X(\{\mu\}, b)$ is connected to a point in $X_w(b)$ for some $\s$-straight element $w \in \Adm(\{\mu\})$. This is the first reduction theorem, established in \S\ref{4}. It is based on the reduction method in \cite{He99} and the quasi-affineness of irreducible components of affine Deligne-Lusztig varieties. 

The quasi-affineness is obtained by constructing certain ample line bundles on affine Deligne-Lusztig varieties. This requires the knowledge of the Picard group and ample line bundles on the affine flag varieties. In the equal characteristic case, the description is due to Pappas and Rapoport in \cite{PR}. In the mixed characteristic case, we work with the Witt vector affine flag variety $\mathcal{FL}=L^+G/L^+\brI$ of Zhu \cite{Zhu} and Bhatt-Scholze \cite{BS}. In this case we have the following result which is of independent interest.

  \begin{theorem}\label{theorem1}Assume $G$ is simply connected. We have $$\text{Pic}\mathcal{FL}\cong \bigoplus_{\iota\in\brbS}\mathbb{Z}[\frac{1}{p}]\mathscr{L}(\epsilon_\iota).$$ Moreover, a line bundle $\mathscr{L}$ corresponding to $\sum\lambda_\iota\epsilon_\iota$ is ample if and only if $\lambda_\iota>0$ for all $\iota\in\brbS$.
   \end{theorem}

Here $\brbS$ is the set of simple reflections in the Iwahori-Weyl group $\brW$. We refer to \S\ref{3} for the definitions and other notations. This is analogous to the result of \cite[Theorem 10.3]{PR} in equal characteristic, however it is proved in a different way by using the method of $h$-descent developed in \cite{BS}. The idea is to descend line bundles from a suitable Demazure resolution of $\mathcal{FL}$, whose Picard group has a simple description. By analyzing this resolution and using \cite[Theorem 7.7]{BS}, one shows that the correct line bundles descend.

Using the above description of $\Pic\mathcal{FL}$ and applying the descent result to the fibration $LG/L^+\brI\rightarrow LG/L^+\brK$, we also get a description of the Picard group of the partial affine flag varieties. In particular, we show that if $\breve K$ is hyperspecial, then $\text{Pic}(G(\breve F)/\breve K) \cong \BZ[\frac{1}{p}]$. This answers a question of Bhatt and Scholze in \cite{BS}.

\subsection{} Our next major step is to connect certain points in $X_w(b)$ for a $\s$-straight element $w$ inside $X(\{\mu\}, b)$. As $J_b$ acts transitively on $X_w(b)$, we only need to connected certain elements of $J_b$. To do this we find an explicit set of generators for $J_b$, this is Theorem \ref{Jb}. Roughly speaking, $J_b$ is generated by a certain subgroup of $\brI$,  and the elements $u_{-J}$. 

We then construct a curve inside $X(\{\mu\}, b)$ that connects $p$ and $u_{-J} \cdot p$, where $p$ is a point in $X_w(b)$. This is based on the comparison of the admissible sets of the (non-standard) Levi subgroups and the whole reductive group, and on G\"ortz's result on the connectedness of classical Deligne-Lusztig varieties. 

Finally in the Appendix we show that the notion of connected components as defined in \cite{CKV} agrees with the notion in the Zariski topology, thus there is no ambiguity when we talk about connected components. Moreover the notion in \cite{CKV} is useful for applications to Shimura varieties and Rapoport-Zink spaces, so it useful to know the two notions coincide. 

\subsection{Acknowledgments} 
We thank M. Kisin for his encouragement and discussions in an early stage of the project. We thank U. G\"ortz for his valuable suggestions on a preliminary version of the paper, which leads to several simplifications. We also thank B. Bhatt, T. Haines, P. Hamacher, S. Nie, A. Patel, M. Rapoport, P. Scholze, B. Smithling, E. Viehmann and X. Zhu for their valuable comments and suggestions. 

\section{Preliminaries}

\subsection{} Let $\BF_q$ be a finite field with $q$ elements. Let $F$ be a non-archimedean local field with valuation ring $\CO_F$ and residue field $\BF_q$. Let $\brF$ be the completion of the maximal unramified extension of $F$ with valuation ring $\CO_{\brF}$ and residue field $\kk=\bar \BF_q$. Let $\G=\Gal(\bar F/F)$ be the absolute Galois group and $\G_0=\Gal(\bar F/F^{un})$ be the inertia subgroup. Let $\s$ be the Frobenius of $\brF$ over $F$.

Let $G$ be a connected reductive group over $F$.  We also write $\s$ for the Frobenius automorphism on $G(\brF)$. Let $\brS \subset G$ be a maximal $\breve F$-split torus defined over $F$ and let $\brT$ be its centralizer. Let $\brI$ be the Iwahori subgroup fixing a $\sigma$-invariant alcove $\breve \fka$ in the apartment $V$ attached to $\brS$. The relative Weyl group $\brW_0$ and the Iwahori-Weyl group $\brW$ are defined by 
\begin{gather*} 
\brW_0=\breve N(\brF)/\brT(\brF), \qquad \brW=N(\brF)/\brT(\brF) \cap \brI,
\end{gather*}
where $\breve N$ denotes the normalizer of $\brS$ in $G$.

The Iwahori-Weyl group $\brW$ is a split extension of $\brW_0$ by the  subgroup $X_*(\brT)_{\G_0}$. The splitting depends on the choice of a special vertex of $\breve{\mathfrak{s}}$. When considering an element $\lambda \in X_*(\brT)_{\G_0}$ as an element of $\brW$, we write $t^\l$. We fix such a special vertex $\breve{\mathfrak{s}}$ which also allows us to identify $V\cong X_*(\brT)_{\Gamma_0}\otimes_{\mathbb{Z}}\mathbb{R}$.
For any $w \in \brW$, we denote by $\dot w$ a representative of $w$ in $\breve N(\brF)$. 

Let $\brW_a$ be the associated affine Weyl group and $\brbS$ be the set of simple reflections associated to $\breve{\mathfrak{a}}$.  Since $\breve{\mathfrak{a}}$ is $\sigma$-invariant, there is a natural action of $\sigma$ on $\brbS$. The Iwahori-Weyl group $\brW$ contains the affine Weyl group $\brW_a$ as a normal subgroup and we have a natural splitting $$\brW=\brW_a \rtimes \Omega,$$ where $\Omega$ is the normalizer of $\breve \fka$ and is isomorphic to $\pi_1(G)_{\G_0}$. 

The length function $\ell$ and the Bruhat order $\le$ on the Coxeter group $\brW_a$ extend in a natural way to $\brW$. 

\subsection{} For any $b \in G(\breve F)$, we denote by $[b]=\{g \i b \s(g); g \in G(\breve F)\}$ its $\s$-conjugacy class. Let $B(G)$ be the set of $\s$-conjugacy classes of $G(\breve F)$. The $\s$-conjugacy classes are classified by Kottwitz in \cite{Ko85} and \cite{Ko97}. We denote by $\bar \nu$ the Newton map $$\bar \nu: B(G) \to (X_*(\brT)^+_{\G_0, \BQ})^\s,$$ where $X_*(\brT)^+_{\G_0, \BQ}$ is the intersection of $X_*(\brT)_{\G_0} \otimes \BQ$ with the set of dominant (rational) coweight in $X_*(\brT) \otimes \BQ$.  Recall Kottwitz \cite{Ko97} has defined a map $\tilde{\kappa}:G(L)\rightarrow \pi_1(G)_{\Gamma_0}$. We denote by $\k$ the map $$\k: B(G) \to \pi_1(G)_\G$$ induced by $\tilde{\k}$ and the composition $\pi_1(G)_{\G_0}\rightarrow \pi_1(G)_\G$. By \cite[\S 4.13]{Ko97}, the map $$(\bar \nu, \k): B(G) \to (X_*(\brT)^+_{\G_0, \BQ})^\s \times \pi_1(G)_\G$$ is injective. 

The maps $\bar \nu$ and $\k$ can be described in an explicit way via the inclusion map $\brW \to G(\brF), w \mapsto \dot w$. For any $w \in \brW$, there exists a positive integer $n$ such that $\s^n$ acts trivially on $\brW$ and such that $w \s(w) \cdots \s^{n-1}(w)=t^\l$ for some $\l \in X_*(\brT)_{\G_0}$. We set $\nu_w=\frac{\l}{n} \in X_*(\brT)_{\G_0, \BQ}$. We denote by $\bar \nu_w$ the unique dominant rational coweight in the $\brW_0$-orbit of $\nu_w$. We denote by $\k(w)$ the image of $w$ under the natural projection map $\brW \to \brW/\brW_a \cong \pi_1(G)_{\G_0} \to \pi_1(G)_\G$. 

We denote by $B(\brW, \s)$ the set of $\s$-conjugacy classes of $\brW$. The inclusion map $\brW \to G(\brF), w \mapsto \dot w$ induces a map $\Psi: B(\brW, \s) \to B(G)$, which is independent of the choice of representatives $\dot w$. By \cite{He99}, the map $\Psi$ is surjective we have a commutative diagram 
\[\xymatrix{B(\brW, \s) \ar@{->>}[rr]^{\Psi} \ar[dr] & & B(G) \ar@{^{(}->}[ld] \\ & (X_*(\brT)^+_{\G_0, \BQ})^\s \times \pi_1(G)_\G &}.\]

The map $B(\brW, \sigma) \to B(G)$ is not injective. However, there exists a canonical lifting to the set of straight conjugacy classes. 

By definition, an element $w \in \brW$ is {\it $\sigma$-straight} if for any $n \in \mathbb N$, $$\ell(w \sigma(w) \cdots \sigma^{n-1}(w))=n \ell(w).$$ It is equivalent to the condition that $\ell(w)=\<\bar \nu_w, 2 \rho\>$, where $\rho$ is the half sum of all positive roots in the reduced root system associated to $\brW_a$. A $\sigma$-conjugacy class is {\it straight} if it contains a $\sigma$-straight element. It is easy to see that the minimal length elements in a given straight $\s$-conjugacy class are exactly the $\s$-straight elements. When the action of $\s$ on $\brW$ is trivial, we will call these straight elements instead of $\sigma$-straight. 

It is proved in \cite[Theorem 3.7]{He99} that 

\begin{theorem}\label{str-bg}
The restriction of $\Psi: B(\brW, \sigma) \to B(G)$ gives a bijection from the set of straight $\sigma$-conjugacy classes of $\brW$ to $B(G)$. 
\end{theorem}

\subsection{} Now we recall some remarkable properties on the minimal length elements in a conjugacy class of $\brW$. 

For $w, w' \in \brW$ and $s \in \brbS$, we write $w \xrightarrow{s}_\s w'$ if $w'=s w \s(s)$ and $\ell(w') \le \ell(w)$. We write $w \to_\s w'$ if there is a sequence $w=w_0, w_1, \cdots, w_n=w'$ of elements in $\brW$ such that for any $k$, $w_{k-1} \xrightarrow{s}_\s w_k$ for some $s \in \brbS$. We write $w \approx_\s w'$ if $w \to_\s w'$ and $w' \to_\s w$. It is easy to see that $w \approx_\s w'$ if $w \to_\s w'$ and $\ell(w)=\ell(w')$. If the action of $\s$ on $\brW$ is trivial, then we will omit $\s$ in the subscript. 

It is proved in \cite[Theorem 2.9 \& Theorem 3.8]{HN14} that

\begin{theorem}\label{HN-min}
Let $\CO$ be a $\s$-conjugacy class of $\brW$. Then 

(1) For any $w \in \CO$, there exists a minimal length element $w'$ of $\CO$ such that $w \to_\s w'$. 

(2) If $\CO$ is straight, then all the $\s$-straight elements in $\CO$ form a single $\approx_\s$-equivalence class. 
\end{theorem}

\section{The affine Deligne-Lusztig variety $X(\{\mu\}, b)_K$}\label{2}

\subsection{} Let $\brK$ be a standard $\sigma$-invariant parahoric subgroup of $G(\brF)$, i.e., a $\s$-invariant parahoric subgroup that contains $\brI$. We denote by $K \subset \brbS$ the corresponding set of simple reflections. Then $\s(K)=K$. We denote by $W_K \subset \brW$ the subgroup generated by $K$. We have $$G(\brF)=\sqcup_{w \in W_K \backslash \brW/W_K} \brK \dot w \brK.$$

The affine Deligne-Lusztig variety is introduced by Rapoport in \cite{R:guide}. For any $w \in W_K \backslash \brW/W_K$ and $b \in G(\brF)$, we set $$X_{K, w}(b)=\{g \breve \CK \in G(\breve F)/\breve \CK; g \i b \s(g) \in \breve \CK \dot w \breve \CK\}.$$ If $\brK=\brI$, we simply write the corresponding affine Deligne-Lusztig variety as $X_w(b)$.

In the equal characteristic case, it is known that the affine Deligne-Lusztig varieties $X_{K,w}(b)$ could be endowed with a scheme structure by viewing it as a locally closed subvariety of a partial affine flag variety. Thanks to \cite[Theorem 10.6]{BS} (see also \cite{Zhu}), we may, even in mixed characteristic, endow $X_{K,w}(b)$ with a scheme structure in the same way. In this case it is a perfect scheme, hence generally not of finite type over $\mathbb{F}_q$ so that the term "variety" is still not completely justified. However all topological notions are well-defined using the Zariski topology, therefore we have notions of dimension, irreducible components and connected components for $X_{K,w}(b)$.

\subsection{} Let $\{\mu\}$ be a conjugacy class of cocharacters of $G$ and $\underline \mu$ be the image in $X_*(\brT)_{\G_0}$ of a dominant (for a choice of Borel defined over $\brF$) representative $\mu$ in $X_*(\brT)$ of the conjugacy class  $\{\mu\}$. The admissible set is defined by $$\Adm(\{\mu\})=\{w \in \brW; w \le t^{x(\underline \mu)} \text{ for some } x \in \brW_0\}.$$ Note that $\Adm(\{\mu\})$ has a unique minimal element with respect to the Bruhat order $\le$, i.e., the unique element $\t_{\{\mu\}}$ in $\Omega$ with $\t_{\{\mu\}} \in t^{\underline \mu} \brW_a$. 

More generally, for any standard parahoric subgroup $\brK$, we set \begin{gather*}
\Adm(\{\mu\})^K=W_K \Adm(\{\mu\}) W_K \subset \brW, \\ \Adm(\{\mu\})_K=W_K \backslash \Adm(\{\mu\})^K/W_K \subset W_K \backslash \brW/W_K.
\end{gather*}

Let \begin{align*} X(\{\mu\}, b)_K &=\{g \in G(\brF)/\brK; g \i b \s(g) \in \cup_{w \in \Adm(\{\mu\})_K} \brK \dot w \brK\} \\ &=\sqcup_{w \in \Adm(\{\mu\})_K} X_{K, w}(b).\end{align*} This is a union of affine Deligne-Lusztig varieties in $G(\brF)/\brK$. If $\brK=\brI$, we simply write  $X(\{\mu\}, b)$. 

In this paper, we are mainly interested in these unions of affine Deligne-Lusztig varieties. They play a crucial role in the study of the reduction of Shimura varieties. 

\subsection{} We recall the nonemptiness pattern of $X(\{\mu\}, b)_K$. 

To do this, we recall the definition of neutral acceptable set $B(G, \{\mu\})$ in \cite{RV}. Note that there is a partial order on the set of dominant elements in $X_*(\brT) \otimes \BQ$ (namely, the {\it dominance order}) defined as follows. For $\l, \l' \in X_*(\brT) \otimes \BQ$, we write $\l \leq \l'$ if $\l'-\l$ is a non-negative rational linear combination of positive relative coroots. Set $$B(G, \{\mu\})=\{ [b]\in B(G); \kappa([b])=\mu^\natural, \bar \nu_b\leq \mu^\diamond \}.$$
 Here $\mu^\natural$ denotes the common image of $\mu\in\{\mu\}$ in $\pi_1(G)_\Gamma$, and $\mu^\diamond$ denotes the Galois average of a dominant representative of the image of an element  of $\{\mu\}$ in $X_*(T)_{\Gamma_0, \BQ}$  with respect to the L-action of $\sigma$ on $(X_*(T)_{\Gamma_0,\BQ})^+$.   The set $B(G, \{\mu\})$ inherits a partial order from $X_*(\brT) \otimes \BQ$. It has a unique minimal element, the $\s$-conjugacy class $[\dot \t_{\{\mu\}}]$. It is proved in \cite{HNx} that $B(G, \{\mu\})$ has a unique maximal element. 

The following result is conjectured by Kottwitz and Rapoport in \cite{KR} and \cite{R:guide} and proved by the first author in \cite{He00}. 

\begin{theorem}\label{KR-He}
Let $\brK$ be a standard parahoric subgroup of $G(\brF)$ and $b \in G(\brF)$. Then 

(1) The set $X(\{\mu\}, b)_K \neq \emptyset$ if and only if $[b] \in B(G, \{\mu\})$. 

(2) The natural projection $G(\brF)/\brI \to G(\brF)/\brK$ induces a surjection $$X(\{\mu\}, b) \twoheadrightarrow X(\{\mu\}, b)_K.$$ 
\end{theorem}

\section{Affine flag varieties in mixed characteristics and line bundles}\label{3}
\subsection{}The aim of the section is to give a description of the Picard group of the mixed characteristic affine flag variety of \cite{BS}. In the equal characteristic setting, the description of the Picard group is well-known, see, e.g. \cite[Proposition 10.1]{PR}. An essential component of the proof there is the  existence of the big cell corresponding to the open $G(\kk[t \i])$-orbit on the affine flag variety. It is pointed out by Bhatt and Scholze in \cite[Question 11.6 (iv)]{BS} that such approach seems to break down in mixed characteristic. Here we develop a different strategy to study the Picard group, based on the relation between the Picard groups of Demazure varieties and of the affine flag variety and the method of $h$-descent developed in \cite{BS}. This approach works for both mixed characteristic and equal characteristic (after taking perfections).

\subsection{Line bundles in equal characteristic} We briefly recall the construction of the affine flag variety in equal characteristic.

Let $\mathbf{k}$ be an algebraically closed field of characteristic $p$. For $G$ a reductive group over $k((t))$, which splits over a tamely ramified extension and $\brI$ an Iwahori subgroup of $G$, Pappas and Rapoport construct an associated ind-projective ind-scheme over $\mathbf{k}$ called the affine flag variety which is defined to be the fpqc quotient $$\mathcal{FL}=L^+G/L^+\brI.$$

For simplicity assume $G$ is simply connected. Then $\mathcal{FL}$ is  connected and we have an isomorphism \cite[Proposition 10.1]{PR}$$\Pic(\mathcal{FL})\cong \bigoplus_{\iota\in\brbS}\mathbb{Z}.$$

The line bundles are constructed by identifying $\mathcal{FL}$ as an inductive limit of the Schubert varieties appearing in Kac-Moody theory \cite[9.27]{PR}. Under this identification, to each affine weight $\lambda$, \cite[XVIII, Proposition 28]{Ma} constructs a line bundle $\mathscr{L}(\lambda)$ on $\mathcal{FL}$. If $\brK_\iota$ is the parahoric subgroup corresponding to $\iota\in\brbS$, the degree of the restriction $\lambda$ to $\mathbb{P}^1\cong L^+\brK_{i}/L^+\brI$ is given by the coefficient of the fundamental weight $\epsilon_\iota$ in $\lambda$, and these degrees suffice to characterize $\mathscr{L}(\lambda)$.

\subsection{} In the rest of this section, $F$ denotes a discrete valuation field  of {\it mixed characteristic} with perfect residue field $k$ and $\mathbf{k}$ denotes an algebraic closure of $k$. 

We recall the Witt vector affine Grassmannian of \cite{BS}, see also \cite{Zhu}. 

Let $\mathcal{G}$ be a smooth affine group scheme over $\CO_F$ with generic fibre $G$. For a $k$-algebra $R$, we define the relative Witt vectors $$W_{\CO_F}(R)=W(R)\otimes_{W(k)}\CO_F.$$

The $p$-adic loop group is the functor on perfect $k$-algebras $R$ given by $$LG(R)=G(W_{\CO_F}(R)[\frac{1}{p}]).$$ 

The positive $p$-adic loop group is the functor on perfect $k$-algebras $R$ given by $$L^+\mathcal{G}(R)=\mathcal{G}(W_{\CO_F}(R)).$$

Clearly these are functors valued in groups and it is known that $L^+\mathcal{G}$ is representable by  a perfect affine scheme and $LG$ is a strict ind-perfect affine scheme. 
In fact the functor $L^+\mathcal{G}$ is the perfection of the scheme $L^+_p\mathcal{G}:=\varprojlim L^h_p\mathcal{G}$ where $L^h_p\mathcal{G}$ is the finite type scheme over $\mathbf{k}$ given by the Greenberg realization of $\mathcal{G}$ over $\mathcal{O}_{\brF}/\pi^h$.

Let $\text{Perf}$ denote the category of perfect schemes over $k$. The affine Grassmannian associated to $\mathcal{G}$ is the functor on $\text{Perf}$ given by the fpqc quotient
$$Gr_{\mathcal{G}}=LG/L^+\mathcal{G}.$$

It is proved in \cite[Corollary 10.6]{BS}  that $Gr_{\mathcal{G}}$ is an increasing union of perfections of quasi-projective schemes over $k$.

We are particularly interested in $Gr_\mathcal{G}$ when $\mathcal{G}$ is a (connected) parahoric group scheme. In this case it is known that $Gr_\mathcal{G}$ is representable by an inductive limit of perfections of projective schemes, cf. \cite[Corollary 10.6]{BS}.
 
We also know that in this case, the Kottwitz homomorphism induces bijections $$
\pi_0(LG)\cong\pi_0(Gr_\mathcal{G})\cong\pi_1(G)_\Gamma,$$
 cf. \cite[Proposition 1.21]{Zhu}. 
  When $\mathcal{G}$ is the connected smooth affine group scheme corresponding to the Iwahori $\breve{\mathcal{I}}$, we will call $LG/L^+\mathcal{G}$ the {\it affine flag variety} for $\mathcal{G}$ and denote it by $\mathcal{FL}.$
  
The main result of this section is the following:
  
  \begin{theorem}\label{theorem1} Assume that $G$ is simply connected and $\text{char}(F)=0$.
  
   (1) There is an isomorphism $$\text{Pic}\mathcal{FL}\cong \bigoplus_{\iota\in\brbS}\mathbb{Z}[\frac{1}{p}],$$ where the isomorphism is given by taking $\mathscr{L}$ to the degree of its restriction to $L\mathcal{P}_\iota^+/L^+\brI\cong\mathbb{P}^{1,p^{-\infty}}$.
   
   (2) A line bundle $\mathscr{L}$ corresponding to $(\lambda_\iota)_{\iota \in \brbS}$ is ample if and only if $\lambda_\iota>0$ for all $\iota\in\brbS$.
   \end{theorem}
   \begin{remark}Recall that the Picard group of a perfect $\mathbb{F}_q$ scheme is always a $\mathbb{Z}[\frac{1}{p}]$ module. This follows since the $q$-Frobenius is an isomorphism and it induces multiplication by $q$ on the Picard group.

  \end{remark}

 \subsection{}We first explain how this theorem can be used to describe the Picard group of the affine flag variety for non-simply connected $G$. Recall that $\mathbf{k}$ is an algebraically closed field. In this case the Kottwitz homomorphism $G(\brF)\rightarrow \pi_1(G)_{\Gamma_0}$ is surjective. For an element $\t\in\pi_1(G)_{\Gamma_0}$, we may take an element $g\in LG(\mathbf{k})$ whose image in $\pi_1(G)_{\Gamma_0}$ is $\t$. Then multiplication by $g$ induces an isomorphism over $\mathbf{k}$ of the connected component of $\mathcal{FL}$ corresponding to $\t$ and the neutral component $\mathcal{FL}^0$ corresponding to $0$.
  
  Let $G_{\text{der}}$ be the derived group of $G$ and $\tilde{G}$ the simply connected covering of $G_{\text{der}}$. Let $\widetilde{\mathcal{FL}}$ be flag variety for $\tilde{G}$ and the corresponding Iwahori subgroup of $\tilde{G}$.
  
 \begin{proposition}\label{prop3}
 The map $L\tilde{G}\rightarrow LG$ identifies $\widetilde{\mathcal{FL}}$ with $\mathcal{FL}^0$.

 \end{proposition}
\begin{proof}
The same proof as \cite[Proposition 6.6]{PR} shows that $\widetilde{\mathcal{FL}}\rightarrow\mathcal{FL}^0$ is a universal homeomorphism. That this is an isomorphism follows from \cite[Corollary A.16]{Zhu}.
\end{proof}

It follows that $\Pic{\mathcal{FL}}=(\Pic\widetilde{\mathcal{FL}})^{\oplus |\pi_1(G)_{\Gamma}|}$, thus in order to describe the Picard group of $\mathcal{FL}$, we may and do assume that $G$ is simply connected.
 
 \subsection{} We now introduce the Demazure resolutions for the affine flag variety $\mathcal{FL}$. 
 
Recall that $G$ is simply connected. So $\brW=\brW_a$ is a coxeter group. For any $j \in \brbS$, let $s_j$ be the corresponding simple reflection and $\mathcal{P}_j$ be the corresponding parahoric subgroup. For $w\in\breve{W}$, let $S_w$ denote the closure of $L^+\brI \dot w L^+\brI/L^+\brI$ in $\mathcal{FL}$. This is the Schubert variety corresponding to $w$. Let $\underline w=(s_{j_1}, \cdots, s_{j_n})$ be a reduced expression of $w$. Let $\mbox{Supp}(w)=\{j_1,\cdots,j_n\}\subset\brbS$ be the support of $w$. It is known that $\mbox{Supp}(w)$ is independent of the reduced word decomposition for $w$. For $i=0, \cdots, n$, define $\underline w_i=(s_{j_1}, \cdots, s_{j_i})$. We form the Demazure variety 
 
 $$D_{\underline w}=L^+\mathcal{P}_{j_1}\times^{L^+\brI}L^+\mathcal{P}_{j_2}\times^{L^+\brI}\cdots\times^{L^+\brI}L^+\mathcal{P}_{j_n}/L^+\brI$$
 as an object in $\text{Perf}$. 
 
For any $j\in \brbS$, the quotient $L^+\mathcal{P}_j/L^+\brI$ is  isomorphic to $\mathbb{P}^{1,p^{\infty}}$. Thus we have a sequence of morphisms $$D_{\underline w}=D_{\underline w_n}\rightarrow D_{\underline w_{n-1}}\rightarrow\cdots\rightarrow D_{\underline w_0},$$ given by forgetting the last coordinate, and where each map is a locally trivial fibration with fiber isomorphic to $\mathbb{P}^{1,p^{-\infty}}$, see for example \cite[Prop. 8.8]{PR}, \cite[1.5.2]{Zhu}. In particular $D_{\underline w}$ is perfectly finitely presented over $\mathbf{k}$, cf.  \cite[Corollary A.23]{Zhu}. 
 
 \begin{proposition}\label{Pic}
$D_{\underline{w}}$ is the perfection of a smooth projective $\mathbf{k}$ scheme and there is an isomorphism  $$\text{Pic}(D_{\underline{w}})\cong \bigoplus_{i=1}^n\mathbb{Z}[\frac{1}{p}],$$
given by taking a line bundle $\mathscr{L}$ to the degree of its restriction to a fibre of $D_{\underline{w}_i}\rightarrow D_{\underline{w}_{i-1}}$.
\end{proposition}

\begin{proof}
We prove this statement for $D_{\tilde{w}_i}$ by induction on $i$. Suppose $D_{\tilde{w}_{i-1}}$ is the perfection of a smooth projective scheme $X$. Then we claim  the $\mathbb{P}^{1,p^{-\infty}}$ bundle $D_{\tilde{w}_i}\rightarrow D_{\tilde{w}_{i-1}}$ arises from a locally trivial $\mathbb{P}^1$ bundle over $Y\rightarrow X^{(m)}$ for some $m$ (here $X^{(m)}$ denotes the $m^{th}$ Frobenius twist of $X$). Indeed since $D_{\tilde{w}_{i-1}}$ is quasi-compact, we can find a  covering $\{U_i\}$ of $D_{\tilde{w}_{i-1}}$ consisting of finitely many Zariski opens such that $D_{\tilde{w}}|_{U_i}\cong \mathbb{P}^{1,p^{-\infty}}\times U_i$. Then we obtain a Cech cocycle $f_{ij}\in PGL_2(U_{ij})$ where $U_{ij}=U_i\cap U_j$. Since $D_{\tilde{w}_{i-1}}\rightarrow X$ is a universal homeomorphism, the $U_i$ descend to Zariski opens $U_{i}'\subset X$. Let $U_{ij}'=U_{i}'\cap U_j'$, the $f_{ij}$ arises from $f_{ij}'\in PGL_2(U_{ij}^{(m)})$ for some $m$ sufficiently large. Twisting by a further power of Frobenius we can assume $f_{ij}'f_{jk}'f_{ki}'=1\in PGL_2(U_{ijk}')$, i.e. $f_{ij}'$ is a cocycle. This defines a $\mathbb{P}^1$ bundle $Y\rightarrow X^{(m)}$ whose perfection is $D_{\tilde{w}}\rightarrow D_{\tilde{w}_{i-1}}$.

Since $X^{(m)}$ is smooth, $Y$ is smooth projective and we have $\Pic(Y)=\Pic(X^{(m)})\oplus\mathbb{Z}$ where the $\mathbb{Z}$ factor is given by the degree of the restriction of a line bundle to a fibre of $Y\rightarrow X^{(m)}$. Taking perfections and using \cite[Lemma 3.5]{BS}, it follows  that $\Pic(D_{\tilde{w}_i})=\Pic(D_{\tilde{w}_{i-1}})\oplus\mathbb{Z}[\frac{1}{p}]$, hence the result follows by induction.
\end{proof}

\begin{remark}\label{remark1}Note the following subtlety. For each $i$ we have maps $$\phi_i: S_{s_{j_i}}=L^+\mathcal{P}_{j_i}/L^+\brI\rightarrow D_{\tilde{w}}$$
given by $\phi_i(x)=(1,\cdots,x,\cdots,1)$ where the $x$ appears in the $i^{th}$ position. The map $\text{Pic}(D_{\underline{w}})\to \bigoplus_{i=1}^n\mathbb{Z}[\frac{1}{p}]$ is given by sending $\mathscr{L}$ to the degree of $\phi_i^*\mathscr{L}$, which is a line bundle on $S_{s_{j_i}}\cong\mathbb{P}^{1,p^{-\infty}}$. However the degree will depend on the isomorphism $\mathbb{P}^{1,p^{-\infty}}\cong S_{s_{j_i}}$, namely  twisting by a power of Frobenius will change the isomorphism by a power of $p$ in one of the factors. It was pointed out to us by G\"ortz that for $h \gg 0$, we have an isomorphism $L_p^h\mathcal{P}_{j_i}/L_p^h\brI\cong \mathbb{P}^1_{\mathbf{k}}$ and taking perfection gives an isomorphism $S_{j_i}\cong \mathbb{P}_{\mathbf{k}}^{1,p^-\infty}$. This isomorphism is well-defined up the action of $\Aut(\mathbb{P}^1_{\mathbf{k}})$, hence gives a canonical identification of $\Pic S_{s_{j_i}}\cong \mathbb{Z}[\frac{1}{p}]$; from now on we fix this isomorphism and  use this in the definition of the isomorphism Theorem \ref{theorem1}.
\end{remark}

 \subsection{} 
 As in \cite{PR}, there is a proper surjective map $$\Psi: D_{\underline w}\rightarrow S_w.$$ Since $S_w$ is perfectly finitely presented over $\mathbf{k}$, so is $D_{\underline w}\rightarrow S_w$. It follows that $$S_w=\bigcup_{v\leq w}L^+\brI vL^+\brI/L^+\brI.$$ In this situation, we may apply the results of \cite{BS} to descend line bundles from $D_{\underline w}$ to $S_w$. The general strategy is given by Theorem 7.7 of loc. cit.
 
 \begin{theorem}[\cite{BS} Theorem 7.7]\label{descent}\footnote{It was communicated to us by Bhatt and Scholze that in the newest version of their paper, this condition can be replaced by the weaker condition that the fibers of $f:X\rightarrow Y$; are geometrically connected.} Let $f : X \rightarrow  Y$ be a proper surjective perfectly finitely presented map in $\mbox{Perf}$ such that $Rf_*\CO_X=\CO_Y$ in
 particular, all geometric fibres of f are connected. Let $\mathscr{L}$ be a vector bundle on $X$. Then $\mathscr{L}$ descends to Y if and only if
 for all geometric points $\overline{y}$ of $Y$, $\mathscr{L}_{\overline{y}}$ is trivial on the fibre $X_{\overline{y}}$.
 \end{theorem}
 
 Now we show that $D_{\underline w}\rightarrow S_w$ satisfies the conditions in Theorem \ref{descent}.
 
 \begin{proposition}\label{directimage}
 Let $w \in \brW$ and $\underline w$ be a reduced expression of $w$. Then $$R\Psi_*\CO_{D_{\underline w}}=\CO_{S_w}.$$
 \end{proposition}
 
 \begin{proof}We use induction on the length of $w$; for $\ell(w)=0$, $\Psi$ is an isomorphism so this holds.
 
 Suppose that $\underline w=(s_{j_1}, \cdots, s_{j_n})$ with $n>0$. Let $w=s_{j_1}w'$ and $\underline w'=(s_{j_2}, \cdots, s_{j_n})$. Then $\underline w'$ is a reduced expression of $w'$ and $\Psi$ factors as $$D_{\underline w}=L^+\mathcal{P}_{j_1}\times^{L^+\brI}D_{\underline w'}\rightarrow L^+\mathcal{P}_{j_1}\times^{L^+\brI}S_{w'}\rightarrow S_w.$$
 
 The first map satisfies the condition on direct images by induction.  This follows from either the base change result \cite[Lemma 3.18]{BS}, or alternatively, it can be checked using the fiberwise criterion \cite[Lemma 7.8]{BS}. The second map is a proper surjective perfectly finitely presented map, and so by \cite[Lemma 7.8]{BS}, it suffices to check for all geometric points $\overline{y}$ of $S_w$ with residue field $k(\overline{y})$ and fibre $D_{\underline w,\overline{y}}$, that $R\Gamma(D_{\underline w,\overline{y}},\CO_{D_{\underline w,\overline{y}}})=k(\overline{y})$.
 
 Let $Z=\bigcup_{w''}S_{w''}\subset S_w$ where the union runs over $w''<w'$, with $s_{j_1}w''<w''$. Then the second map is an isomorphism away from $Z$ and the fibre over a point $\overline{y}$ in $Z$ is given by $\mathbb{P}_{k(\overline{y})}^{1,p^{-\infty}}$. The result then follows.
 \end{proof}

Note that $\Psi \circ \phi_{j_i}$ is the natural embeddings of $$S_{s_{j_i}}=L\mathcal{P}_{j_i}^+/L^+\brI\hookrightarrow S_w.$$ Then

 \begin{lemma}\label{lemma4} Let $\mathscr{L}$ be  a line bundle on $S_w$ and suppose $\Psi^*\mathscr{L}$ corresponds to $(\lambda_i)$ under the isomorphism in Proposition \ref{Pic}. Then $\l_i=\deg(\mathscr{L} \mid_{S_{s_{j_i}}})$. 
 \end{lemma}

 \begin{proposition}\label{prop1}
 We have a isomorphism
 $$\mbox{Pic}(S_w) \xrightarrow{\cong}\bigoplus_{\iota\in\mbox{Supp(w)}}\mathbb{Z}[\frac{1}{p}], \qquad \mathscr{L} \mapsto (\deg(\mathscr{L} \mid_{S_{s_\iota}}))_{\iota \in \mbox{Supp(w)}}.$$ 
 \end{proposition}
 
 \begin{proof} Let $\underline w=(s_{j_1}, \cdots, s_{j_n})$ be a reduced expression of $w$. The map $\Psi: D_{\underline w} \to S_w$ induces a map $\Psi^*: \mbox{Pic}(S_w) \to \mbox{Pic}(D_{\underline w}) \cong \oplus_{i=1}^n \BZ[\frac{1}{p}]$. By \cite[Proposition 7.1]{BS}, $\Psi^*$ is injective. By Lemma \ref{lemma4}, the map is given by $\mathscr{L} \mapsto (\deg(\mathscr{L} \mid_{S_{j_i}}))_{i=1, \cdots, n}$. Therefore the map in the statement of the proposition is an injection.
 
 It remains to show that the map  is surjective. In other words, let $\l_i \in \BZ[\frac{1}{p}]$ for $i=1, \cdots, n$ such that $\lambda_r=\lambda_s$ for $j_r=j_s$ and $\tilde {\mathscr{L}} \in \mbox{Pic}(D_{\underline w})$ that corresponds to $(\l_1, \cdots, \l_n)$, we need to show that $\tilde{\mathscr{L}}$ descents to a line bundle on $S_w$. 

We argue by induction on $l(w)$. It is clearly true when $l(w)=1$, i.e. $w$ is a simple reflection. Suppose it is surjective for all $w'$ with $l(w')< w$. 
 For $i=0,1,\cdots,n$, define $X_i=D_{\underline w_i}\times^{L^+\brI}S_{s_{j_{i+1}}\cdots s_{j_n}}$. We have $X_n=D_{\underline w}$ and $X_0=S_w$ and there is a sequence of morphisms $$X_n\rightarrow X_{n-1}\rightarrow \cdots\rightarrow X_0$$ induced by multiplication $L\mathcal{P}_{j_i}\times^{L^+\brI}S_{s_{j_{i+1}}\cdots s_{j_n}}\rightarrow S_{s_{j_i}\cdots s_{j_n}}$.
 
We show that $\tilde{\mathscr{L}}$ descends to a line bundle $\mathscr{L}_i$ on $X_i$ for each $i$ by descending induction. This is tautologically true for $i=n$ so suppose $\tilde{\mathscr{L}}$ descends to $X_{i+1}$. To show $\tilde{\mathscr{L}}$ descends to $X_i$, it suffices by Theorem \ref{descent} to  check  the restriction of $\mathscr{L}_{i+1}$ to any fiber of $X_{i+1}\rightarrow X_i$ is trivial.  
 
 Given $(p_1,\cdots ,p_i)\in L^+\mathcal{P}_{j_1}\times^{L^+\brI}\cdots \times^{L^+\brI}L^+\mathcal{P}_{j_i}$, we  define a map $$\alpha:S_{s_{j_{i+1}}\cdots s_{j_n}}\rightarrow X_i, \qquad s\mapsto (p_1,\cdots ,p_i,s).$$
 
 Upon pulling back $X_{i+1}\rightarrow X_i$ along $\alpha$ we obtain a Cartesian diagram:
 \[\xymatrix{Y \ar[r]^-\gamma \ar[d]_-\b & S_{s_{j_{i+1}}\cdots s_{j_n}} \ar[d]^-\a\\
 X_{i+1} \ar[r] & X_i}
 \]
where $\g$ can be identified with the multiplication map  $L\mathcal{P}_{j_{i+1}}\times^{L^+\brI}S_{s_{j_{i+2}}\cdots s_{j_n}}\rightarrow S_{s_{j_{i+1}}\cdots s_{j_n}}$.

As $(p_1,\cdots ,p_i)$ varies, the maps $\alpha$ cover $X_i$, hence it suffices to prove that $\beta^*\mathscr{L}_{i+1}$ descends along $\gamma$. Upon relabelling we reduce to the case $$f:X_1=L^+\mathcal{P}_{j_1}\times^{L^+\brI}S_{w'}\rightarrow S_w=X_0$$
where $s_{j_1}w'=w$ and $l(w)>l(w')$. The projection $X_1\rightarrow L^+\mathcal{P}_{j_1}/L^+\brI=S_{j_1}\cong\mathbb{P}^{1,p^{-\infty}}$ exhibits $X_1$ as an $S_{w'}$ bundle over $\mathbb{P}^{1,p^{-\infty}}$, hence we have an isomorphism $\Pic(X_1)\cong \Pic(S_{w'})\oplus \mathbb{Z}[\frac{1}{p}]$, see \cite[Theorem 5]{Mag}. By induction hypothesis on $l(w')$, we have $$\mbox{Pic}(X_1) \cong \left(\bigoplus_{\iota\in\text{Supp}(w')}\mathbb{Z}[\frac{1}{p}]\right)\oplus\mathbb{Z}[\frac{1}{p}].$$ Note that by our assumption on $\mathscr{L}_1$, if $j_i=j_1$ for some $i>1$, then $$\deg(\mathscr{L}_1 \mid_{L^+\mathcal{P}_{j_1}\times^{L^+\brI} L^+\brI})=\l_1=\l_i=\deg(\mathscr{L}_1 \mid_{L^+\brI \times^{L^+\brI} S_{j_i}}).$$
 
 The map $X_{1}\rightarrow X_0$ is proper perfectly finitely generated and by the proof of Proposition \ref{directimage} we have $$Rf_*\CO_{X_{1}}\cong \CO_{X_0}.$$
 
 Therefore by Theorem \ref{descent}, to show that $\mathscr{L}_1$ descends, it suffices to check that the restriction of $\mathscr{L}_1$ to each geometric fibre is trivial. As in  Proposition \ref{directimage}, $f$ is an isomorphism away from $S_{w''}$ where $w''<w'$ and $s_{j_1}w''<w''$, so the condition is satisfied away from this locus. We have the following fibre diagram:
\[\xymatrix{
L^+\mathcal{P}_{j_1}\times^{L^+\brI}S_{w''} \ar[r]^-h \ar[d]_g & S_{w''} \ar[d]^-i\\
L^+\mathcal{P}_{j_1}\times^{L^+\brI}S_{w'} \ar[r]^-f & S_w
}\] 

We will show that $g^*\mathscr{L}_1$ descends to a line bundle on $S_{w''}$. 

By induction we know $$\mbox{Pic}S_{w''}\cong \bigoplus_{\iota\in\text{Supp}(w'')}\mathbb{Z}[\frac{1}{p}](\epsilon_{\iota}).$$

Furthermore since $h: L^+\mathcal{P}_{j_1}\times^{L^+\brI}S_{w''} \rightarrow S_{w''}$ is a $\mathbb{P}^{1,p^{-\infty}}$ bundle, it follows that $$\mbox{Pic} (L^+\mathcal{P}_{j_1}\times^{L^+\brI}S_{w''})\cong \left(\bigoplus_{\iota\in\text{Supp}(w'')}\mathbb{Z}[\frac{1}{p}]\right)\oplus\mathbb{Z}[\frac{1}{p}].$$

By our construction, the map $h^*:\mbox{Pic}S_{w''}\rightarrow\mbox{Pic}L^+\mathcal{P}_{j_1}^+\times^{L^+\brI^+}S_{w''}$ is given by
 $$\bigoplus_{\iota\in\text{Supp}(w'')}\mathbb{Z}[\frac{1}{p}]\rightarrow\left(\bigoplus_{\iota\in\text{Supp}(w'')}\mathbb{Z}[\frac{1}{p}]\right)\oplus\mathbb{Z}[\frac{1}{p}], \quad (\lambda_\iota) \mapsto((\lambda_\iota), 0).$$
We need to check that $g^*\mathscr{L}_1$ is of this form, ie. its restriction to a fiber of $h$ is trivial. The fiber of $h$ over the base point in $e\in S_{w''}$ is given by the image of the map $$ \alpha:L^+\mathcal{P}_{j_1}/L^+\brI\rightarrow L^+\mathcal{P}_{j_1}/L^+\brI\times^{L^+\brI}L^+\mathcal{P}_{j_1}/L^+\brI$$ given by $p\mapsto (p,p^{-1})$. As above we have an isomorphism $\Pic L^+\mathcal{P}_{j_1}/L^+\brI\times^{L^+\brI}L^+\mathcal{P}_{j_1}/L^+\brI\cong \mathbb{Z}[\frac{1}{p}]\oplus\mathbb{Z}[\frac{1}{p}]$. By our assumption $g^*\mathscr{L}$ corresponds to $(\lambda_1,\lambda_1)\in\mathbb{Z}[\frac{1}{p}]\oplus\mathbb{Z}[\frac{1}{p}]$ under the above isomorphism. 

Now if we consider the line bundle $\mathcal{O}(\lambda_1)\in \Pic(\mathbb{P}^{1,p^{\infty}})\cong \Pic(L^+\mathcal{P}_{j_1}/L^+\brI)$, then this pulls back to $g^*\mathscr{L}$ under the multiplication map $m:L^+\mathcal{P}_{j_1}/L^+\brI\times^{L^+\brI}L^+\mathcal{P}_{j_1}/L^+\brI$. Since the composition $m\circ\alpha$ contracts $L+\mathcal{P}_{j_1}/L^+\mathcal{P}_{j_1}$ to a point, the restrction of $g^*\mathscr{L}$ to the fiber of $h$ over $e$ is trivial.
 
 It follows that the restrictions of $\mathscr{L}_1$ to all geometric fibres are trivial, and hence $\mathscr{L}_1$ descends to $S_w$. The surjectivity of the map in the Proposition is proved. 
 \end{proof}
 
 \subsection{Proof of Theorem \ref{theorem1} (1)}
 Since $\mathcal{FL}$ is an increasing union of $S_w$ as $w$ ranges over $\brW$, in order to prove Theorem \ref{theorem1} it suffices to show that the description of the Picard groups of $S_w$ in Proposition \ref{prop1}  is compatible with the natural inclusions $S_{w'}\rightarrow S_w$ for $w'<w$. This follows from the following commutative diagram
\[\xymatrix{\coprod_{\iota\in\text{Supp}(w')}S_{s_\iota} \ar[r] \ar[d]_-h & S_{w'} \ar[d]\\
\coprod_{\iota\in\text{Supp}(w) }S_{s_\iota} \ar[r] & S_w}\]
where the horizontal maps induce the isomorphisms of Picard groups and $h^*$ is the natural map $\bigoplus_{{\iota}\in\text{Supp(w)}}\mathbb{Z}[\frac{1}{p}]\rightarrow\bigoplus_{\iota\in\text{Supp(w')}}\mathbb{Z}[\frac{1}{p}]$.

\subsection{}For the proof of Theorem \ref{theorem1} (2), we follow the strategy of \cite[\S 9.3]{BS}. The proofs of loc. cit. goes through in our situation with some minor changes using the Demazure resolution $D_{\underline w}\rightarrow S_w$.

We denote by $(\mathscr{L}(\epsilon_\iota))_{\iota \in \brbS}$ the basis elements of $\text{Pic}(\mathcal{FL})$ obtained by the isomorphism in Theorem \ref{theorem1} (1). Similarly, for any $w \in \brW$ and its reduced expression $\underline w$, we denote by $(\mathscr{L}(a_i))_{i=1, \cdots, n}$ the basis element of $\text{Pic}(D_{\underline w})$ obtained by the isomorphism in Proposition \ref{Pic}. 

For $w\in\brW$, let $O_w$ be the $L^+\brI$ orbit of the point $\dot{w}L^+\brI$ in $\mathcal{FL}$. For $i=1,\cdots ,n$, let $C_i$ be the closed subscheme of $D_{\underline w}$ where the $i^{th}$ term is in $L^+\brI $. Let $\partial D_{\underline w}=\bigcup_{i=1}^n C_i$ be the boundary of $D_{\underline w}$. Then the map $D_{\underline w}\backslash \partial D_{\underline w}\rightarrow O_w$ is an isomorphism.

\begin{lemma}\label{lemma2} (1) The line bundle $\otimes_{i=1}^n \mathscr{L}(a_i)^{\lambda_i}$ is ample if $\l_1\gg \l_2\gg\cdots \gg \l_n\gg 0$.

(2) For $x\in D_{\underline w}\backslash\partial D_{\underline w}$, there exists a section $s\in \mbox{H}^0(D_{\underline w},\mathscr{L}(a_i))$ such that $s(x)\neq 0$.
\end{lemma} 
\begin{proof}
(1) Recall we have the sequence of $\mathbb{P}^{1.p^{-\infty}}$ fibrations $$D_{\underline w_n}\rightarrow D_{\underline w_{n-1}}\rightarrow \cdots \rightarrow D_{\tilde{w_0}}.$$

We assume inductively there exists $\mu_1\gg\mu_2\gg\cdots \gg \mu_{m-1}\gg 0$ such that $\otimes_{i=1}^{m-1}\mathscr{L}(a_i)^{\mu_i}$ is ample on $D_{\underline w_{m-1}}.$ Since $D_{\underline w_{m}}\rightarrow D_{\underline w_{m-1}}$ is a $\mathbb{P}^{1,p^{-\infty}}$ fibration and $\mathscr{L}(a_m)$ is relatively ample for this morphism, we have $\mathscr{L}(a_m)\otimes(\otimes_{i=1}^{m-1}\mathscr{L}(a_i)^{\mu_i})^N$ is ample for $N\gg 0$. The result then follows by induction.

(2) Recall we have a $\mathbb{P}^{1,p^{-\infty}}$ fibration $D_{\underline w_i}\rightarrow D_{\underline w_{i-1}}$ and hence a line bundle $\CO_i(1)$ on $D_{\underline w_i}$. The line bundle $\mathscr{L}(a_i)$ is given by the pullback of $\CO_i(1)$ along the map $D_{\underline w_n}\rightarrow D_{\underline w_{i}}$.

We have a section of $\CO_i(1)$ corresponding to the section $D_{\underline w_{i-1}}\rightarrow D_{\underline w_i}$ defined by $(x_1,\cdots ,x_{i-1})\rightarrow (x_1,\cdots ,x_{i-1},1)$. The pullback of this section along $D_{\underline w_n}\rightarrow D_{\underline w_i}$ gives a section of $\mathscr{L}(a_i)$ which does not vanish on $D_{\underline w}\backslash\partial D_{\underline w}$.
\end{proof}

\subsection{Proof of Theorem \ref{theorem1} (2)} Let $\mathscr{L}=\otimes\mathscr{L}(\epsilon_\iota)^{\lambda_\iota}$ be a line bundle on $S_w$ with $\lambda_\iota>0$ for all $\iota$ and let $\mathscr{M}$ be its pullback to $D_{\underline w}$. We prove that $\mathscr{L} \mid_{S_w}$ is ample by induction on $\ell(w)$. 

We show that 

(a) \quad $\mathscr{L}$ is strictly nef.

 This can be proved in the same way as \cite[Lemma 9.10]{BS}. Indeed let $C_{\text{perf}}$ be the perfection of a smooth projective curve over $k$ and let $f:C_{\text{perf}}\rightarrow S_w$ be a non-constant map. Wlog. we may assume $f$ maps the generic point of $C_{\text{perf}}$ into $O_w$. As $D_{\underline w}\rightarrow S_w$ is surjective, the map $C_{\text{perf}}\rightarrow S_w$ lifts generically to $D_{\underline w}$, and since $D_{\underline w}\rightarrow S_w$ is proper, $f$ lifts to $\tilde{f}:C_{\text{perf}}\rightarrow D_{\underline w}$.

Since $C_{\text{perf}}$ meets $D_{\underline w}\backslash \partial D_{\underline w}$, the pullback of $\mathscr{L}(a_i)^{\lambda_{j_i}}$ to $C_{\text{perf}}$ has a non-vanishing section, hence it is effective and therefore has non-negative degree. Since $f^*\mathscr{L}=\otimes_{i=1}^n\tilde{f}^*\mathscr{L}(a_i)^{\lambda_{j_i}}$, if this has degree 0, then $\tilde{f}^*\mathscr{L}(a_i)^{\lambda_{j_i}}$ has degree 0 for all $i$, hence it is trivial. However by \ref{lemma2} a weighted product of $\mathscr{L}(a_i)^{\lambda_{j_i}}$ is ample on $D_{\underline w}$, hence has positive degree on $C_{\text{perf}}$. Therefore the pullback of $\mathscr{L}$ to $C_{\text{perf}}$ has positive degree. This proves (a).

We show that 

(b) \quad $\mathscr{L}$ is semiample.

As in the proof of \cite[Lemma 9.11]{BS} we have $\mathscr{M}$ is big with exceptional locus $E(\mathscr{M})$ contained in the boundary $\partial D_{\underline w}=\Psi^{-1}(\bigcup_{w'<w}S_{w'})$. Since $\mathscr{M}$ is nef, by \cite[Theorem 1.9]{K} it suffices to show $\mathscr{M}|_{E(\mathscr{M})}$ semiample. By  induction we have that $\mathscr{L}|_{S_{w'}}$ is ample for all $w'<w$. By \cite[Lemma 1.8]{K}, we have $\mathscr{L}|_{\bigcup_{w'<w}S_{w'}}$ is ample and hence $\mathscr{M}|_{\partial D_{\underline w}}$ is semiample. Since $D_{\underline w}\rightarrow S_w$ has connected fibres, $\mathscr{L}$ is semiample. This proves (b).

Combining (a) and (b), we have that $\mathscr{L}$ is ample.

\subsection{Line bundles on partial affine flag varieties}

Let $K\subset\brbS$ be any subset and let ${\brK}$ denote the corresponding parahoric subgroup. We can use the above results to get a description of the Picard groups of the partial affine flag varieties $Gr_{\brK}$ corresponding to ${\brK}$. We continue to assume  $G$ is simply connected. 

As in \cite[Proposition 8.7 a)]{PR}  we have a closed immersion $L^+\brI\rightarrow L^+{\brK}$. The identity map on $G(\brF)$ induces a surjection: $$\Phi:\mathcal{FL}\rightarrow Gr_{{\brK}}.$$

Let $\brW^K$ be the set of minimal length representatives of $\brW/\brW_K$. For $w\in \brW^K$, let $S^{{\brK}}_w\subset Gr_{{\brK}}$ denote the closure of $O_w^{{\brK}}=L^+\brI\dot{w}L^+{\brK}/L^+{\brK}$ in $Gr_{{\brK}}$. Since $S^{{\brK}}_w$ is the image of $S_w$ under $\Phi$, we have $$S^{{\brK}}_w=\bigcup_{w'\in\brW^K, w'\leq w}L^+\brI\dot{w}L^+{\brK}/L^+{\brK}$$
and $Gr_{{\brK}}$ is the rising union of the $S_w^{{\brK}}$. For $w\in \brW^K$, the map $$h:\bigcup_{u\in \brW_K}S_{wu}\rightarrow S_{w}^{{\brK}}$$ is a fibration with fiber $L^+{\brK}/L^+\brI$, which admits sections locally for the \'etale topology, see eg. \cite[Proof of Lemma 1.3]{Zhu}.

Let $\overline{{\brK}}^{\,red}$ denote the reductive quotient of the special fibre of ${\brK}$. As in \cite[Proposition 8.7 b)]{PR} and \cite[Corollary 1.24]{Zhu}, we have $L_p^+\brI$ is the preimage in $L_p^+{\brK}$ of a (perfection of a) Borel subgroup in $\overline{B}$ in $\overline{{\brK}}^{\,red}$. Thus the quotient $L_p^+{\brK}/L_p^+\brI$ can be identified with $\overline{\brK}^{red}/\overline{B}$ a finite flag variety. Taking perfections, we obtain an identification of $$L^+{\brK}/ L^+\brI$$ with the perfection of a finite flag variety.

\begin{proposition} (1) We have an isomorphism:
$$\mbox{Pic}(Gr_{{\brK}})\cong \bigoplus_{\iota\in\brbS-K}\mathbb{Z}[\frac{1}{p}],$$
where the isomorphism is given taking $\mathscr{L}$ to the  degree of its restriction along $\mathbb{P}^{1,p^{-\infty}}\cong S_{s_\iota}\cong S_{s_\iota}^{{\brK}}\rightarrow Gr_{{\brK}}$.

(2) A line bundle $\mathscr{L}=\otimes_{\iota\in\brbS-K}\mathscr{L}(\epsilon_\iota)^{\lambda_\iota}$ is ample if and only if $\lambda_\iota>0$ for all $\iota\in\brbS-K$.
\end{proposition}

\begin{proof} (1) We show $\mathscr{L}=\otimes_{\iota\in\brbS}\mathscr{L}(\epsilon)^{\lambda_\iota}\in\Pic(\mathcal{FL})$ descends to $Gr_{{\brK}}$ if and only if $\lambda_{\iota}=0$ for $\iota\in K$.

Let $w\in \brW^K$. We apply the fibral criterion for descent of vector bundles Theorem \ref{descent} to the fibration $\bigcup_{u\in \brW_K}S_{wu}\rightarrow S_{w}^{{\brK}}$. The map is proper surjective perfectly finitely presented since both $\bigcup_{\brW_K}S_{wu}$ and $S_{w}^{{\brK}}$ are, and the condition $Rh_*\mathcal{O}_{\bigcup_{u\in \brW_K}S_{wu}}=\mathcal{O}_{S_{w}^{{\brK}}}$ can be checked fiberwise by \cite[Theorem 7.8]{BS}. We need to check for $\overline{y}$ a geometric point of $S_{w}^{{\brK}}$ and $U_{\overline{y}}$ the fiber of $h$ over $\overline{y}$, that $R\Gamma(U_{\overline{y}},\mathcal{O}_{U_{\overline{y}}})=k(\overline{y})$. Since $U_{\overline{y}}$ is the perfection of a finite flag variety, this follows from the theorem of Borel-Weil-Bott.

Let $\mathscr{L}=\otimes_{\iota\in\brbS}\mathscr{L}(\epsilon)^{\lambda_\iota}$, by Theorem \ref{descent}, $\mathscr{L}$ descends to $Gr_{\brK}$ if and only if $\mathscr{L}|_{U_{\overline{y}}}$ is trivial. Thus if $\mathscr{L}|_{U_{\overline{y}}}$, descends to $Gr_{\brK}$, we have $ \lambda_{\iota}=0$ for $\iota\in K$ since $S_{s_\iota}$ maps to a point in $Gr_{\brK}$. 

For the converse, suppose $\lambda_\iota=0$ for all $\iota\in K$. We show that $\mathscr{L}|_{U_{\overline{y}}}$ for all $\overline{y}$ a geometric point of $X$.

Let $\pi:V\rightarrow S_w^{\brK}$ be an \'etale covering such that the fiber bundle $\bigcup_{u\in \brW_K}S_{wu}\rightarrow S_{w}^{{\brK}}$ splits. We have a pullback diagram:
 \[\xymatrix{X\times_{\mathbf{k}} V \ar[r] \ar[d]_q& V \ar[d]^-\a\\
 \bigcup_{u\in \brW_K}S_{wu} \ar[r] & S_w^{\brK},}
 \]
where $X$ is isomorphic to the perfection of a finite flag variety. Since $H^1(X,\mathcal{O}_X)=0$ by Borel-Weil-Bott, we have $\Pic(X\times_{\mathbf{k}}V)\cong\Pic(V)\times \Pic(X)$, the isomorphism is given by restricting  a line bundle to the fibers of the projections $p_1:X\times_{\mathbf{k}}V\rightarrow X$ and $p_2:X\times_{\mathbf{k}}V\rightarrow V$. Let $\overline{x}$ be the base point of $S_{w}^{\brK}$, i.e. corresponds to the image of 1 in $S_w^{\brK}$. By the description of the Picard group of finite flag varieties and our assumptions on $\mathscr{L}$, we have $\mathscr{L}|_{U_{\overline{x}}}$ is trivial. Thus by the above, the restriction of $q^*\mathscr{L}$ to any fiber of $X\times_{\mathbf{k}}V\rightarrow V$ is trivial. Since $V\rightarrow S^{\brK}_w$ is surjective, $\mathscr{L}|_{U_{\overline{y}}}$ is trivial for all geometric points $\overline{y}$ of $S_w^{\brK}$.

(2) Let $\mathscr{M}$ be any ample line bundle on $S_w^{{\brK}}$. Note that $\mathscr{M}=\otimes_{\iota\in\brbS-K}\mathscr{L}(\epsilon_\iota)^{\delta_\iota}$ with $\delta_\iota>0$. Indeed  $S_{s_\iota}^{{\brK}}\cong S_{s_\iota}\cong\mathbb{P}^{1,p^{-\infty}}$, and so the restriction of $\mathscr{M}$ to $S_{s_\iota}^{{\brK}}$ has positive degree, but this degree is just $\delta_\iota$.
  
  Now let $\mathscr{L}=\otimes_{\iota\in\brbS-K}\mathscr{L}(\epsilon)^{\lambda_\iota}$ where $\lambda_\iota>0$, we will show $\mathscr{L}$ is ample using the same strategy as the proof of Theorem \ref{theorem1} (2).
  
  (a) \quad $\mathscr{L}$ is strictly nef.
 
 Let $C_{\text{perf}}$ be the perfection of a smooth projective curve and  $C_{\text{perf}}\rightarrow S_{w}^{{\brK}}$ any map. As before we may assume $C_{\text{perf}}$ intersects $O_w^{{\brK}}$. Again $\mathscr{L}(\epsilon_\iota)$ has a non-vanishing section on $O_w^{{\brK}}$ since its pullback to $S_w$ has a non-vanishing section on $O_w$, hence $\mathscr{L}(\epsilon_\iota)|_{C_{\text{perf}}}$ is effective. Since a positive combination of the $\mathscr{L}(\epsilon_\iota)$ is ample, $\mathscr{L}\mid_{C_{\text{perf}}}$ has positive degree.
 
 (b) \quad $\mathscr{L}$ is semiample.
 
  We  prove that $\mathscr{L}|S^{{\brK}}_w$ is semiample for $w\in \brW^K$ by induction on $\ell(w)$. Indeed this true for $l(w)=1$, i.e. $w=s_\iota\in\brbS-K$. Now assume it's true for all $w'\in\brW^K$, with $l(w')<l(w)$. 
  
  Let $\mathscr{M}=\otimes\mathscr{L}(\epsilon)^{\delta_\iota}$ be an ample line bundle on $Gr_{{\brK}}$. Upon raising $\mathscr{L}$ to a sufficiently large power, we may assume $\lambda_\iota>\delta_\iota$ and hence $\mathscr{L}|_{S_w^{{\brK}}}$ is the tensor product of an ample line bundle with an effective one, i.e. $\mathscr{L}\mid_{S_w^{{\brK}}}$ is big. Moreover, since $\mathscr{L}$ has a non-vanishing section on $O_w^{{\brK}}$, the exceptional locus lies in $S_w^{{\brK}}\backslash O_w^{{\brK}}=\bigcup_{w'\in \brW^K,w'<w}S_{w'}^{{\brK}}$. By induction $\mathscr{L}$ is semiample on $S_{w'}^{{\brK}}$ for all $w' \in \brW^K$ with $w'<w$, hence by \cite[Lemma 1.8]{K}, $\mathscr{L}\mid_{\bigcup_{w'\in \brW^K,w'<w}}S_{w'}^{{\brK}}$ is semiample. Thus $\mathscr{L}$ is semiample.
  
The result now follows from (a) and (b).
\end{proof}

\begin{remark}
In the case when ${\brK}$ is a special parahoric subgroup, the proposition gives an isomorphism $\mbox{Pic}Gr_{\brK}\cong\mathbb{Z}[\frac{1}{p}]$. This answers a question of Bhatt and Scholze \cite[Question 11.6 (iii)]{BS}.
\end{remark}

\section{First reduction theorem}\label{4}

In this section we study the connected components of $X(\{\mu\},b)$. As in the introduction  $F$ will be a non-archimedean local field of {\it any} characteristic. In the equal characteristic case we make the assumption that the group $\G$ splits over a tamely ramified extension of $F$ and that char $k$ does not divide $\pi_1(G_{\text{ad}})$, where $G_{\text{ad}}$ is the adjoint group of $G$.

\subsection{} For $w\in\brW$ we write $$X_{\leq w}(b)=\bigcup_{w'\leq w}X_{w'}(b).$$
In general, for any finite subset $C$ of $\brW$ we  write $$X_C(b)=\bigcup_{w\in C}X_w(b).$$ 
If moreover $C$ is closed under the Bruhat order,  then $X_C(b)$ is a closed subscheme of the affine flag variety, in particular it is closed and hence projective. The main theorem of this section is the following:
\begin{theorem}\label{theorem2}
Let $C$ be a finite subset of $\breve W$ that is closed under the Bruhat order and $Y$ be an irreducible component of $X_C(b)$. Then $Y\cap X_x(b)\neq\emptyset$ for some $\sigma$-straight element $x$ in $C$.
\end{theorem}

In the rest of this paper, we will mainly use the following special case of Theorem \ref{theorem2}. 

\begin{corollary}\label{cor1}
Every point in $X(\{\mu\},b)$ lies in the same Zariski  connected component as a point in $X_x(b)$ for some $\s$-straight element $x\in \Adm(\{\mu\})$.
\end{corollary}

\subsection{Reduction to adjoint groups}\label{red-1} For $b\in G(\breve{F})$ and $\t\in \pi_1(G)_{\Gamma}$, we let $b_{\text{ad}}$ and $\t_{\text{ad}}$ denote their images in $G_{\text{ad}}(\brF)$ and $\pi_1(G_{\text{ad}})_{\Gamma}$ respectively. Choosing maximal tori in $G$ and $G_{\text{ad}}$ compatibly, we obtain a map $\brW\rightarrow \brW_{\text{ad}}$ denoted $x\mapsto x_{\text{ad}}$. The alcove $\breve{\mathbf{a}}$ determines an alcove $\breve{\mathbf{a}}_{\text{ad}}$ in the building of $G_{\text{ad}}$ and we let $\brI_{\text{ad}}$ be the corresponding Iwahori subgroup.

Let $\mathcal{FL}_{\text{ad}}$ be the flag variety for $G_{\text{ad}}$. For $\lambda\in\pi_1(G)_\Gamma$ let $\mathcal{FL}^\lambda$ be the connected component of $\mathcal{FL}$ corresponding to $\t$ and similarly for $\t_{\text{ad}}$. We have the following proposition which can be proved in the same way  as \cite[Proposition 2.2.1]{GHN}.

\begin{proposition}
The map $G\rightarrow G_{\text{ad}}$ induces an isomorphism $$\mathcal{FL}^{\t,\text{red}}\cong \mathcal{FL}_{\text{ad}}^{\t_{\text{ad}}},$$
where $\mathcal{FL}^{\t,{\text{red}}}$ is the reduced locus of $\mathcal{FL}^{\t}$.
\end{proposition}
\begin{remark}When $F$ is equi-characteristic and $G$ is not semisimple, $\mathcal{FL}$ is not reduced, so we must take  the reduced locus. In the mixed characteristic situation this issue does not arise since perfect rings are always reduced. Since we are only interested in connected components, the non-reducedness makes no difference to us.
\end{remark}
 
 For an affine Deligne Lusztig variety $X_w(b)$ we let $X_w(b)^{\t}=X_w(b)\cap\mathcal{FL}^\t$, and similarly for $G_{\text{ad}}$. The previous proposition implies 
 
 \begin{corollary}\label{adjoint}
 The isomorphism $\mathcal{FL}^{\t}\cong\mathcal{FL}^{\t_{\ad}}_{\text{ad}}$ induces an isomorphism $$X_w(b)^{\t}\cong X_{w_{\text{ad}}}(b_{\text{ad}})^{\t_{\text{ad}}}.$$
 \end{corollary}
 
Since $w\mapsto w_{\text{ad}}$ is compatible with the property of $\sigma$-straightness, it suffices to prove Theorem \ref{theorem2} for $G_{\text{ad}}$. 

\subsection{Reduction to simply connected groups}\label{red-2} We may write $w_{\text{ad}}$ as $w_{\text{ad}}=w' \t$ for $w' \in \brW_a$ and $\t \in \Omega_{\text{ad}}$. Let $H_{\text{ad}}$ be the inner form of $G_{\text{ad}}$ with $G_{\text{ad}}(\brF)=H_{\text{ad}}(\brF)$ and the Frobenius on $H_{\text{ad}}$ is given by $\sigma'=\text{int}(\dot{\tau})\circ\sigma$. Then we have an identification of buildings $B(G_{\text{ad}},\brF)\cong B(H_{\text{ad}},\brF)$. Since $\tau\in\Omega$, the alcove $\breve{\mathbf{a}}$ corresponds to a $\sigma'$-invariant alcove in $B(H_{\text{ad}},\brF)$ and it's corresponding Iwahori subgroup can be identified with $\breve{\mathcal{I}}_{\text{ad}}$. It is then straightforward to check that the natural identification $G_{\text{ad}}(\brF)=H_{\text{ad}}(\brF)$ induces an isomorphism $$X^{G_{\text{ad}}}_{w_{\text{ad}}}(b)\cong X^{H_{\text{ad}}}_{w'}(b \dot \t \i).$$

Let $H$ be the simply connected cover of $H_{\text{ad}}$ and $\pi: H \to H_{\text{ad}}$ be the projection. Since $w' \in \brW_a$, $X^{H_{\text{ad}}}_{w'}(b \dot \t \i)=\emptyset$ unless $\k_{H_{\text{ad}}} (b \dot \t \i)=1$. In the latter case, we may replace $b \dot \t \i$ by an element $\pi(b')$, where $b' \in H$. 

We have $X^{H_{\text{ad}}}_{w'}(\pi(b'))=\sqcup_{\g \in \Omega_{\text{ad}}} X^{H_{\text{ad}}}_{w'}(\pi(b'))^\g$. Moreover, $X^{H_{\text{ad}}}_{w'}(\pi(b'))^\g \cong X^H_{\g w' \g \i}(b')$. 

As a summary, we have the identification $X^G_w(b) \cong \sqcup_{\g} X^H_{\g w' \g \i}(b')$. It is easy to check that $w$ is $\s$-straight if and only if $\g w' \g \i$ is $\s'$-straight for some (or equivalently, any) $\g \in \Omega_{\text{ad}}$. 

Thus to prove Theorem \ref{theorem2}, we only need to consider simply connected groups. 

\

We have the following result on the dimension of irreducible components of affine Deligne-Lusztig varieties. 

\begin{proposition}\label{prop2}
Let $w \in \brW$ and $b \in G(\brF)$ such that $X_w(b) \neq \emptyset$. If $w$ is not $\s$-straight, then $\dim Y>0$ for any irreducible component $Y$ of $X_w(b)$.
\end{proposition}
\begin{proof}
We follow the strategy of \cite[Theorem 6.1]{He99}. If $w$ is not minimal length in it's $\sigma$-conjugacy class, then by Theorem \ref{HN-min} (1), there exists $w' \in \brW$ and $s\in\brbS$ such that $w \approx_\s w'$ and $sw'\sigma(s)<w'$. By the reduction method \`a la Deligne and Lusztig (\cite[Proposition 4.2]{He99}) we have:

\begin{itemize}
\item  $X_{w'}(b)$ is universally homeomorphic to $X_w(b)$.

\item For any irreducible component of $Y_1$ of $X_{w'}(b)$, there exists an irreducible component $Y_2$ of $X_{sw'}(b)$ or $X_{sw'\sigma(s)}(b)$, such that $\dim Y_1>\dim Y_2$.
\end{itemize}

Therefore $\dim Y>0$ for any irreducible component $Y$ of $X_w(b)$. 

Now suppose $w$ is of minimal length in its $\sigma$-conjugacy class. By \cite[Theorem 3.5]{He99}, we have $X_w(b)\neq\emptyset$ if and only if $\bar \nu_b=\overline{\nu}_w$ and $\kappa(b)=\kappa(w)$. By \cite[Theorem 4.8]{He99}, every irreducible component of $X_w(b)$ is of dimension $\ell(w)-\langle\overline{\nu}_b,2 \rho\rangle$. In particular, the dimension is $0$ if and only if $\ell(w)=\langle\overline{\nu}_b,2 \rho\rangle$,  i.e. $w$ is $\sigma$-straight. 
\end{proof}

\subsection{} Following \cite[\S 9.6]{DL}, we view $X_w(b)$ as in the intersection in $\mathcal{FL}\times \mathcal{FL}$ of the graph of the map $b\sigma$ and the $L^+G$ orbit $O(w)$ of the point $(w,1)$.
We write $p_1,p_2:\mathcal{FL}\times\mathcal{FL}\rightarrow \mathcal{FL}$ for the projection maps. Note that these projections give $O(w)$ the structure of an  $O_w$ bundle over $\mathcal{FL}$.

The Iwahori-Weyl group $\brW$ acts on $\Pic\mathcal{FL}$. In the equal characteristic situation this follows from the identification of $\Pic\mathcal{FL}$ with the set of fundamental affine weights for the Kac-Moody algebra.  In mixed characteristic, there is no such identification, however one may still define such an action using the Cartan matrix $(A_{ij})_{i,j\in\brbS}$ of the Iwahori-Weyl group. Explicitly, let $s_{i}$ be a simple reflection corresponding to $i\in\brbS$. We let 
$$s_i\mathscr{L}(\epsilon_i)=\mathscr{L}(\epsilon_i)-\sum_{j\in\brbS}A_{ij}\mathscr{L}(\epsilon_j).$$
We then extend it linearly to an action on $\Pic\mathcal{FL}$. We have the following result.

\begin{proposition}\label{lemma3} Let $\mathscr{L}$ be a line bundle on $\mathcal{FL}$ and $w\in\mathcal{FL}$.  Then the restriction to $O(w)$ of the line bundle $p_1^*(w\mathscr{L})\otimes p_2^*\mathscr{L}^{-1}$ on $\mathcal{FL}\times\mathcal{FL}$ is trivial.
\end{proposition}

\begin{proof}
Let $s_1\cdots s_n$ be a reduced word decomposition for $w$. We use the interpretation of $O(w)$ as the twisted product $O(s_1)\tilde{\times}\cdots\tilde{\times}O(s_n)$ as in \cite[\S A.1.3]{Zhu}. Consider the line bundle $(\mathscr{L}_1,-\mathscr{L}_2)\tilde{\times}\cdots\tilde{\times}(\mathscr{L}_n,\-\mathscr{L}_{n+1})$ (here $(\mathscr{L}_i,-\mathscr{L}_{i+1})$ denotes the line bundles $p_1^*\mathscr{L}_i\otimes p_2^*\mathscr{L}^{-1}$ on $O(s_i)$). This corresponds to the line bundle $(\mathscr{L}_1,-\mathscr{L}_{n+1})=p_1^*\mathscr{L}_1\otimes p_2^*\mathscr{L}_{n+1}^{-1}$ on $O(w)$. Setting $\mathscr{L}_{i}=s_i\mathscr{L}_{i+1}$ and $\mathscr{L}_{n+1}=\mathscr{L}$ we obtain the line bundle $p_1^*(w\mathscr{L})\otimes p_2^*\mathscr{L}$ on $O(w)$. Thus it suffices to consider the case where $w=s_i$ is a simple reflection.

Since $O(s_i)\rightarrow \mathcal{FL}$ is fibration with fiber $\mathbb{A}^{1,p^{-\infty}}$, we have an isomorphism $p_2^*: \Pic\mathcal{FL}\cong \Pic O(s_i)$. Let $O(s_i)^{s_j}$ denote the pullback of $p_2:O(s_j)\rightarrow \mathcal{FL}$ along $S_{s_j}\rightarrow\mathcal{FL}$. Then it suffices to check for all $s_j\in\brbS$, the restriction of $p_1^*\mathscr{L}$ and $ p_2^*s_i\mathscr{L}$ to $O(s_i)^{s_j}$ are isomorphic. When  $s_j$ and $s_i$ generate a finite subgroup of $\brW$, this follows from \cite[\S9.6]{DL}. Indeed in this case all computations take place within the perfection of $\overline{\brK}^{red}/\overline{B}\times \overline{\brK}^{red}/\overline{B}\subset\mathcal{FL}\times\mathcal{FL}$. Here $\overline{\brK}^{red}$ is the reductive quotient of the special fiber of the parahoric corresponding to $\{s_i,s_j\}$ and $\overline{B}$ is the image of $\brI$, so that $\overline{\brK}^{red}/\overline{B}$ is finite flag variety.

Thus we only need to consider the case where $\{s_i,s_j\}$ generate an infinite subgroup of $\brW$. In this case the affine root system generated by $s_i,s_j$ is of type $A_1$ or $C$-$BC_1$ (here we use the notation of Tits' table). We do the calculation for type $A_1$, the case of type $C$-$BC_1$ is similar.

For the case of type $A_1$, we may reduce to the case $G=SL_2$. Let $\brI, \mathcal{P}_i,\mathcal{P}_j$ denote the subgroups of $SL_2(\mathcal{O}_{\brF})$ consisting of matrices of the form
$$\brI=\left(\begin{matrix}
 * & *\\ p* & *
\end{matrix}\right), \ \ \ \mathcal{P}_i=\left(\begin{matrix}
* & *\\ * & *
\end{matrix}\right),\ \ \ \mathcal{P}_j=\left(\begin{matrix}
* & p^{-1}*\\ p* & *
\end{matrix}\right)$$where $*$ denotes an element of $\mathcal{O}_{\brF}$.
All calculations take place within the product of Schubert varieties $S_{s_js_i}\times S_{s_j}$. We have an isormorphism $S_{s_js_i}\cong L^+\mathcal{P}_j\times^{L^+\mathcal{\brI}}L^+\mathcal{P}_i/L^+\brI$ and the projection onto $L^+\mathcal{P}_j/L^+\brI$ presents $L^+\mathcal{P}_j\times^{L^+\mathcal{\brI}}L^+\mathcal{P}_i/L^+\brI$ as a $\mathbb{P}^{1,p^{-\infty}}$ bundle over $\mathbb{P}^{1,p^{-\infty}}$. We prove that 

(a)  $L^+\mathcal{P}_j\times^{L^+\mathcal{\brI}}L^+\mathcal{P}_i/L^+\brI$ is isomorphic to the perfection of the Hirzebruch surface $F_2$. The section $L^+\mathcal{P}_j\rightarrow L^+\mathcal{P}_j\times^{L^+\mathcal{\brI}}L^+\mathcal{P}_i/L^+\brI$ is the directrix.

We verify (a) by direct computation using coordinates. Indeed we identify $L^+\mathcal{P}_j/L^+\brI$ with $\mathbb{P}^{1,p^{-\infty}}$ by using the coordinates: $$\left(\begin{matrix}
a & b \\ c & d
\end{matrix}\right)\mapsto [a\bmod p: c\bmod p].$$ 
For $(p_1,p_2)\in L^+\mathcal{P}_j\times^{L^+\mathcal{\brI}}L^+\mathcal{P}_i/L^+\brI$, the matrix $B=p_1p_2$ is well-defined up to left multiplication by an element of $\brI$. Let $x, y$ be the  entries $B_{2,1}$ and $B_{2,2}$. One checks over the open affine subset $c\neq 0$, the coordinates $$\frac{a}{c}\bmod p\times [py\bmod p, \frac{a}{c}y-x\bmod p ]$$ gives a well-defined trivialisation of this $\mathbb{P}^{1,p^{-\infty}}$ bundle. Similarly over $a\neq 0$, we have a trivialisation given by $$\frac{c}{a}\bmod p\times[px\bmod p, \frac{c}{a}x-y\bmod p].$$
The gluing isomorphism over the fiber of a point $[a:c]$ is given by $[s:t]\mapsto [s:\frac{a^2}{c^2}t]$, hence this is the perfection of the Hirszebruch surface $F_2$. The second part of the claim follows easily from the description using these coordinates.

Now (a) is proved.

We write $\mathscr{L}(\epsilon_i)$ (resp $\mathscr{L}(\epsilon_j)$) for the line bundle which has degree 1 (resp. 0) on $S_{s_i}$ and degree 0 (resp. 1) on $S_{s_j}$.

We first check the case $\mathscr{L}=\mathscr{L}(\epsilon_i)$. Then $p_2^*\mathscr{L}_{O(s_i)^{s_j}}$ is trivial, it remains to show that $p_1^*s_i\mathscr{L}(\epsilon_i)$ is trivial on $O(s_i)^{s_j}$. We have the two divisors $S_{s_i}$, $S_{s_j}$ inside the Hirzebruch surface $S_{s_js_i}$ which are generators for the divisor class group. The intersection pairing for this basis has matrix 
 $$\left(\begin{matrix} 0 & 1 \\ 1 & -2\end{matrix}\right).$$ 
We have $s_i\mathscr{L}(\epsilon_i)=-\mathscr{L}(\epsilon_i)+2\mathscr{L}(\epsilon_j)$. Since $\deg s_i\mathscr{L}(\epsilon_i)|_{S_{s_i}}=-1$ and $\deg s_i\mathscr{L}(\epsilon_i)|_{S_{s_i}}=2$, we have $s_i\mathscr{L}(\epsilon_i)=\mathcal{O}_{S_{s_js_i}}(-S_{s_j})$. Thus $s_i\mathscr{L}(\epsilon_i)$ has a non-zero section on $S_{s_js_i}-S_{s_j}$. Since $O(s_i)^{s_j}$ does not intersect $S_{s_j}\times S_{s_j}$, we have $p_1^*s_i\mathscr{L}(\epsilon_i)$ is trivial on $O(s_i)^{s_j}$.

We then check the case $\mathscr{L}=\mathscr{L}(\epsilon_j)$, then $\mathscr{L}(\epsilon_j)$ is the pullback of $\mathscr{L}(\epsilon_j)$ along $S_{s_js_i}\rightarrow S^\flat_{s_js_i}$. We have $s_i\mathscr{L}(\epsilon_j)=\mathscr{L}(\epsilon_j)$ and both $p_1^*\mathscr{L}(\epsilon_j)_{O(s_i)^{s_j}}$ and $p_2^*\mathscr{L}(\epsilon_j)_{O(s_i)^{s_j}}$ have sections vanishing on the divisor $O_{s_i}\times 1\subset O(s_i)^{s_j}$. Hence these line bundles are isomorphic. 

Since $\mathscr{L}(\epsilon)$ and $\mathscr{L}(\epsilon_j)$ form a basis for $\Pic \mathcal{FL}$, this proves the result for type $A_1$.
\end{proof}

\begin{remark}
To be precise, in order to carry out the calculation using intersection theory in the previous Lemma, one must work with a deperfection of the Schubert varieties $S_{s_js_i}$. It can be checked that the deperfections corresponding to the choice of coordinates in a) is compatible with the deperfections $S_{s_i}\cong \mathbb{P}^{1,p^{-\infty}}$ in Remark \ref{remark1}. In particular the line bundles $\mathscr{L}(\epsilon_i)$ and $\mathscr{L}(\epsilon_j)$ come from pullback from the deperfection $F_2$ of $S_{s_js_i}$. One may then perform the calculations over $F_2$ to obtain the result.
\end{remark}

\begin{proposition}\label{prop1}
Let $w \in \brW$ and $b \in G(\brF)$ such that $X_w(b) \neq \emptyset$. Let $Y$ be an irreducible component of $X_w(b)$. If $\text{char}(F)>0$, then $Y$ is a quasi-affine variety. If $\text{char}(F)=0$, then $Y$ is a the perfection of a quasi-affine variety.
\end{proposition}

\begin{proof}
By \S \ref{red-1} and \S \ref{red-2}, it suffices to consider the case where $G$ is simply connected. By \cite{PR} (for $\text{char}(F)>0$) and Theorem \ref{theorem1} (for $\text{char}(F)=0$), we have $$\text{Pic}(\mathcal{FL})=\begin{cases} \bigoplus_{\iota\in\brbS}\mathbb{Z}(\epsilon_i), & \text{ if } \text{char}(F)>0; \\ \bigoplus_{\iota\in\brbS}\mathbb{Z}[\frac{1}{p}](\epsilon_i), & \text{ if } \text{char}(F)=0.\end{cases}$$ The Frobenius $\sigma$ induces an action on $\mbox{Pic}\mathcal{FL}$ which can be identified with action of $\sigma$ on $\brbS$ multiplied by a factor of $q$. 

Since any $\s$-conjugacy class of $G$ is represented by an element in $\brW$, we may assume that $b=\dot x$ for some $x \in \brW$. Then in the group $\brW \rtimes \<\s\>$, we have $x \s w \i=x (\s w \i x) x \i$ and for sufficiently divisible $n$, we have $(x \s w \i)^n=x \s^n t^{n \nu_{w \i x}} x \i=\s^n t^{n x(\nu_{w \i x})}$. In particular, $(x \s w \i)^n$ acts on $\text{Pic}(\mathcal{FL})$ as $q^n t^{n x(\nu_{w \i x})}$ and therefore has no eigenvalue $1$. So the action of $x \s w \i-1$ on $\text{Pic}(\mathcal{FL})_{\BQ}$ is invertible. 

In particular, there exists $\mathscr{L}\in \mbox{Pic}\mathcal{FL}$ such that $(x \s w \i)(\mathscr{L})-\mathscr{L}$ is dominant and regular (i.e. it is of the form $\oplus\lambda_\iota\epsilon_\iota$ with $\lambda_\iota>0$ for all $\iota$). The restriction to $X_w(b)$ of the corresponding line bundle $p_1^*\mathscr{L}\otimes p_2^*(w^{-1}\mathscr{L})^{-1}$ is ample. The statement then follows from the Lemma \ref{lemma3}. 
\end{proof}

\subsection{Proof of Theorem \ref{theorem2}}
Let $x$ be a minimal length element in $\{w \in C; Y\cap X_{w}(b)\neq\emptyset\}$. Then $Y$ contains an irreducible component $Y_1$ of $X_x(b)$.

Let $\overline{Y}_1$ be the closure of $Y_1\in X_C(b)$. By minimality of $x$, we have $\overline{Y}_1\cap X_{<x}=\emptyset$, hence $\overline{Y}_1=Y_1$ and $Y_1$ is projective. By proposition \ref{prop1}, $Y_1$ is also quasi-affine. Thus $\dim Y=0$. It then follows from proposition \ref{prop2} that $x$ is a $\sigma$-straight element.

\section{Structure of $\s$-centralizer}

\subsection{} Let $b \in G(\brF)$ and $J_b=\{g \in G(\brF); g \i b \s(g)=b\}$ be the $\s$-centralizer of $b$. Then $J_b$ acts by left multiplication on $X_w(b)$ for any $w\in \brW$. It is proved in \cite[Theorem 3.5 \& Theorem 4.8]{He99} that

\begin{theorem}\label{jb-X}
Let $w \in \brW$ be a $\s$-straight element and $b \in G(\brF)$. Then $X_w(b) \neq \emptyset$ if and only if $\bar \nu_b=\bar \nu_w$ and $\k(b)=\k(w)$. In this case, $J_b$ acts transitively on $X_w(b)$. 
\end{theorem}

\subsection{}\label{Mv} In this section we study the structure of $\sigma$-centralizer group $J_b$. 

We first describe the standard Levi subgroup associated to a given $\s$-conjugacy class. It is based on the notion of $P$-alcove elements introduced by G\"ortz, Haines, Kottwitz and Reuman in \cite{GHKR} for split groups and generalized in \cite{GHN} for quasi-split groups. We use here the reformation given by Nie in \cite{Ni}. It is stated for unramified (and hence quasi-split groups), but works for arbitrary reductive groups as well. 

Let $\Phi$ be the set of (relative) roots of $\G$ over $\brF$ with respect to $\brS$ and $\Phi_a$ the set of affine roots. Let $\text{Aff}(V)$ be the group of affine transformations on the apartment $V$. The roots in $\Phi$ determine hyperplanes in $V$ and the relative Weyl group $\brW_0$ can be identified with the subgroup of $\text{Aff}(V)$ generated by the reflections through these hyperplanes. For $a \in \Phi$, we denote by $U_a \subset G$ the corresponding root subgroup.

The Frobenius morphism $\s$ on $G(\brF)$ induces an affine morphism on $V$ (see \cite[1.10]{Ti}), which we still denote by $\sigma$. We denote by $\varsigma \in GL(V)$ the linear part of $\sigma$ with respect to the decomposition $\text{Aff}(V)=V \rtimes GL(V)$. 

For any $v \in V$, we set $\Phi_{v, 0}=\{a \in \Phi; \<a, v\>=0\}$ and $\Phi_{v, +}=\{a \in \Phi; \<a, v\>>0\}$. Let $M_v \subset G(\brF)$ be the Levi subgroup generated by $\brT$ and $U_a(\brF)$ for $a \in \Phi_{v, 0}$ and $N_v \subset G(\brF)$ be the unipotent subgroup generated by $U_a(\brF)$ for $a \in \Phi_{v, +}$. Set $P_v=M_v N_v$. Then $P_v$ is a semistandard parabolic subgroup and $M_v$ is a Levi subgroup of $P_v$. Here semistandard means that the parabolic subgroup contains $\brT$. 

Now we give several definitions. Let $w \in \brW$. 

\begin{itemize}
\item We say that $w$ is {\it fundamental} if $\brI \dot w \brI$ is a single $\sigma$-conjugacy class of $\brI$. 

\item We say that $w$ is a {\it $(v, \sigma)$-alcove element} if 

(1) the linear part of $w \circ \sigma \in \text{Aff}(V)$ fixes $v$;

(2) $N_v \cap \dot w \brI \dot w \i \subset N_v \cap \brI$.\footnote{The second condition stated in \cite{Ni} is $w \breve \fka \ge_a \breve \fka$ for $a \in \Phi_{v, +}$. The equivalence of that condition with our condition (2) above is explained in \cite[\S 4.1]{GHN}.}

\item We say that $w$ is {\it $(v, \sigma)$-fundamental} if $w$ is a $(v, \sigma)$-alcove element and $\dot w \sigma(\breve I \cap M_v) \dot w \i=\breve I \cap M_v$. 
\end{itemize}

Note that the condition (1) implies that $\Ad(\dot w) \circ \s$ stabilizes $M_v$. Thus if $w$ is $(v, \s)$-fundamental, then $\Ad(\dot w)$ gives a group automorphism on $\brW_v$ that preserves the set of simple reflections of $\brW_v$. Here $\brW_v=\{w \in \brW; w(v)=v\}$ is the Iwahori-Weyl group of $M_v$. It is worth mentioning that the set of simple reflections $\brbS_v$ of $\brW_v$ is determined by the Iwahori subgroup $\brI \cap M_v$ of $M_v$ and in general, $\brbS_v \not\subset \brbS$.

\smallskip

We have the following result. 

\begin{theorem}\label{fundamental}
Let $w \in \brW$. The following conditions are equivalent:

(1) The element $w$ is $\sigma$-straight;

(2) The element $w$ is fundamental;

(3) The element $w$ is a $(v, \sigma)$-fundamental element for some $v \in V$. 
\end{theorem}

The result is proved by Nie in \cite{Ni} for unramified groups. The general case follows from the same argument. 

\subsection{}\label{brW-t}  Let $[b] \in B(G)$. By Theorem \ref{str-bg}, there exists a $\s$-straight element $w \in \brW$ with $\dot w \in [b]$. We may take $b=\dot w$. Set $M=M_{\nu_w}$ and $\brI_M=\brI \cap M(\brF)$. By Theorem \ref{fundamental}, $w$ is $(v, \s)$-fundamental. Set $\t=\Ad(\dot w) \circ \s$. Then $\t$ induces a group automorphism on $M$ and a length-preserving action on $\brW(M)$. In particular, $\brI_M^\t$ is the Iwahori subgroup of $M(\brF)^\t$. 

By \cite[Remark 6.5]{Ko85} and \cite[\S 4.3]{Ko97}, $J_b=M(\brF)^\t$. By \cite[Lemma 1.6]{Ri}, the Iwahori-Weyl group (over $F$) of $J_b$ is $\brW_{\nu_w}^\t$. For any $w \in \brW^\t$, we denote by $n_w \in G(\breve F)$ a $\t$-stable representative of $w$. 

We have $\brW_{\nu_w}=\brW_{a, \nu_w} \rtimes \breve \Omega_{\nu_w}$, where $\brW_{a, \nu_w}$ is the affine Weyl group associated to $M$ and $\breve \Omega_{\nu_w} \subset \brW_{\nu_w}$ is the subgroup of length-zero elements. Since $\t$ is length-preserving, we have $\t(\brW_{a, \nu_w})=\brW_{a, \nu_w}$ and $\t(\breve \Omega_{\nu_w})=\breve \Omega_{\nu_w}$. In particular, 

(a) \quad $\brW_{\nu_w}^\t=\brW_{a, \nu_w}^\t \rtimes \breve \Omega_{\nu_w}^\t.$ 

The following result is proved in \cite[Theorem A.8]{Lu}. 

\begin{theorem}\label{brWa}
Let $\brbS_{\nu_w}(\t)$ be the set of $\t$-orbits $J$ on $\brbS_{\nu_w}$ with $W_J$ finite. For any $J \in \brbS_{\nu_w}(\t)$, let $w^0_J$ be the longest element of the finite Coxeter group $W_J$. Then $\brW_{a, \nu_w}^\t$ is the affine Weyl group generated by simple reflections $w^0_J$ with $J \in \brbS_{\nu_w}(\t)$. 
\end{theorem}

\begin{remark}
The Braid relations between the simple reflections $w_0^J$ of $\brW_{a, \nu_w}^\t$ is given in \cite[\S A.7]{Lu}. In this paper, we only need the fact that $\brW_{a, \nu_w}^\t$ is generated by $w^0_J$. 
\end{remark}

Using the classification of Dynkin diagrams in \cite[4.2]{BT}, we see that $J \in \brbS_{\nu_w}(\t)$ can be only be of the following two forms:

A) $J=\{\alpha,\tau(\alpha),\cdots \}$ where no combination of any two roots add up to another affine root.

B) $J=\coprod_{i=1}^{n-1}\{\tau^i(\a),\tau^{n+i}(\a)\}$ decomposes into  disjoint union of pairs such that some combination of $\alpha,\tau^n(\a)$ add up to another affine root, but no combination of roots in disjoint pairs add up to another root.

Note the cyclic action of order $n$ on type $A_n$ does not occur since $W_J$ is infinite in that case.

\subsection{} For any affine root $\a$ of $M$, the corresponding affine root subgroup $\CU_\a$ is a group scheme over $\mathcal{O}_{\brF}$ and it follows from \cite[4.1]{BT} that $\mathcal{U}_\alpha(\mathcal{O}_{\brF})$ is a finite free $\mathcal{O}_{\brF}$-module. Similarly we let $\mathcal{U}_{\alpha+}$ denote the union of subgroups $\mathcal{U}_{\alpha+\epsilon}$ for all $\epsilon>0$. Since $\brF$ is discretely valued, $\mathcal{U}_{\alpha+}$ has the structure of a group scheme and we have $\mathcal{U}_\alpha(\mathcal{O}_{\brF})/\mathcal{U}_{\alpha+}(\mathcal{O}_{\brF})$ is a $1$-dimensional vector space over $\overline{\mathbb{F}}_p$. 

Now we construct certain element $u_{-J}$ in $J_b$ for $J \in\brbS_{\nu_w}(\t)$. 

We first assume $J$ is of type A) with $J=\{\alpha,\tau(\a),\cdots ,\tau^{n-1}(\a)\}$. Then $\t^n$ induces  $\sigma^n$-linear automorphisms of $\mathcal{U}_{-\alpha}(\mathcal{O}_{\brF})$ and $\mathcal{U}_{-\alpha+}(\mathcal{O}_{\brF})$, and we obtain a $\sigma^n$ linear automorphism of the quotient $\mathcal{U}_{-\alpha}(\mathcal{O}_{\brF})/\mathcal{U}_{-\alpha+}(\mathcal{O}_{\brF})$. Let $F_n$ denote the fixed field of $\brF$, then by \cite[5.1.17]{BT}, $\mathcal{U}_{-\alpha}$ (resp. $\mathcal{U}_{-\alpha+}$) arises from a group scheme $\mathcal{U}^{\tau^n}_{-\alpha}$ (resp. $\mathcal{U}^{\tau^n}_{-\alpha+}$) over $\mathcal{O}_{F_n}$ and moreover these are group schemes associated to a finite free $\mathcal{O}_{F_n}$ module, in particular $\mathcal{U}^{\tau^n}_{-\alpha}$ and $\mathcal{U}^{\tau^n}_{-\alpha}$ are connected. 

It follows from \cite[5.1.18]{BT}, that $\mathcal{U}^{\tau^n}_\alpha(\mathcal{O}_{F_n})$ surjects onto $[\mathcal{U}_{-\alpha}(\mathcal{O}_{\brF})/\mathcal{U}_{-\alpha+}(\mathcal{O}_{\brF})]^{\tau^n}$. So take $u\in \mathcal{U}^{\tau^n}_\alpha(\mathcal{O}_{F_n})$ mapping to a non-zero element in $[\mathcal{U}_{-\alpha}(\mathcal{O}_{\brF})/\mathcal{U}_{-\alpha+}(\mathcal{O}_{\brF})]^{\tau^n}$ and  define $u_{-J}$ to be the product $$u_{-J}=u\tau(u)\cdots \tau^{n-1}(u) $$

Since no combination of $\tau^i(\a),\tau^j(\a)$ add up to another root for $j\neq i,$ it follows that $\tau^iu$ and $\tau^ju$ commute, hence $u_{-J}$ is fixed by $\tau$, i.e. $u_{-J}\in J_b$.

For the case of type B), let $J=\{\alpha,\tau(\a), \cdots,\tau^{2n-1}(\a)\}$ where $\tau^i(\a)+\tau^{n+i}(\a)$ is an affine root. By the classification in \cite[4.2]{Ti}, this can only occur when the connected component of the Dynkin diagram containing $\alpha$ is of type $A_n$ for $n>2$, $B_n,\ BC_n,\ C_n,\ CB_n$ or $D_n$ for $n\geq 4$. We may apply the same consideration as before to the group scheme $\mathcal{U}_{-\alpha-\tau^n(\a)}$. We obtain a group scheme $\mathcal{U}_{-\alpha-\tau^n(\a)}^{\tau^n}$ which is connected and we let $u\in\mathcal{U}_{-\alpha-\tau^n(\a)}^{\tau^n}(\mathcal{O}_{F_n})$ which maps to a non-zero element in $[\mathcal{U}_{-\alpha-\tau^n(\a)}(\mathcal{O}_{\brF})/\mathcal{U}_{-\alpha-\tau^n(\a)+}(\mathcal{O}_{\brF})]^{\tau^n}$. We define $u_{-J}$ to be $u\tau(u)\cdots\tau^{n-1}(u)$ and as before we see $u_{-J}\in J_b$.

\smallskip

Now we describe the structure of $J_b$. 

\begin{theorem}\label{Jb} Let $w$ be a $\s$-straight element in $\brW$ and $\t=\Ad(\dot w) \circ \s$ be the corresponding group automorphism on $M(\brF)$. Then the $\s$-centralizer $J_{\dot w}$ is generated by $(\brI \cap M_{\nu_w}(\brF))^\t$, $u_{-J}$ for $J \in \brbS_{\nu_w}(\t)$ and $n_w$ for $w \in \breve \Omega_{\nu_w}^\t$.
\end{theorem}

\begin{proof}
We first obtain the Bruhat decomposition of $J_{\dot w}$. 

(a) \quad $J_{\dot w}=\sqcup_{w \in \brW_M^\t} \brI_M^\t n_w \brI_M^\t$. 

Recall that $M(\brF)=\sqcup_{w \in \brW_M} \brI_M \dot w \brI_M$. For any $w \in \brW_M$, $\t(\brI_M \dot w \brI_M)=\brI_M \dot \t(w) \brI_M$. Thus $$J_b=M(\brF)^\t=\sqcup_{w \in \brW_M^\t} (\brI_M \dot w \brI_M)^\t.$$

Let $w \in \brW_M^\t$ and $j \in (\brI_M n_w \brI_M)^\t$. We write $j$ as $j=h_1 n_w h_2$ for some $h_1,h_2\in \brI_M$. 

Let $\brI_w=n_w^{-1}\brI_M n_w\cap \brI_M$. Then $h_2$ is determined up to multiplication by an element of $\brI_w$, hence its image in $\brI_M/\brI_w$ is well defined. Thus $\tau(h_2)h_2^{-1}\in \brI_w$. By \cite[3.1]{BT}, $\brI_w$ is connected. By Lang's theorem on $\brI_w$, upon modifying $h_2$ by an element of $\brI_w$, we may assume $h_2$ is fixed by $\tau$. It follows that $h_1$ is also fixed by $\tau$. Thus $(\brI_M n_w \brI_M)^\t=\brI_M^\t n_w \brI_M^\t$. (a) is proved. 

By (a), $J_{\dot w}$ is generated by $\brI_M^\tau$ and $n_w$ for $w \in \brW_M^\t$. By our construction, $$u_{-J} \in \sqcup_{w \in W_J} \brI_M \dot w \brI_M \cap M(\brF)^\t=\sqcup_{w \in W_J^\t} \brI_M^\t n_w \brI_M^\t$$ Since $W_J^\t=\{1, w^0_J\}$ and $u_{-J} \notin \brI_M$, we have that $u_{-J} \in \brI_M^\t n_{w^0_J} \brI_M^\t$ Therefore $n_{w^0_J}$ is contained in the subgroup of $J_b$ generated by $\brI_M^\t$ and $u_{-J}$. The theorem then follows from \S\ref{brW-t}(a) and Theorem \ref{brWa}.
\end{proof}

\section{Connected components for basic $\s$-conjugacy class}

The first obstruction to connect points in $X(\{\mu\}, b)$ is given by the Kottwitz homomorphism which parametrizes the connected components of $\mathcal{FL}$. 

Let $[b]\in B(G,\{\mu\})$, then $\kappa(g)=\mu^\natural$ in $\pi_1(G)_\G$, hence there exists $c_{b,\{\mu\}}\in \pi_1(G)_{\G_0}$ such that $c_{b,\{\mu\}}-\sigma(c_{b,\{\mu\}})=[\mu]-\tilde{\k}(g)$, where $[\mu]$ is the class of $\{\mu\}$ in $\pi_1(G)_{\G_0}$. Note that the choice of $c_{b,\{\mu\}}$ is determined up to multiplication by an element of $\pi_1(G)_{\G_0}^\sigma$, hence the coset $c_{b,\{\mu\}}\pi_1(G)_{\G_0}^\sigma$ depends only on the class $[b]\in B(G,\{\mu\})$ and $\{\mu\}$. 
\begin{lemma}\label{1-obstr}
Let $[b] \in B(G, \{\mu\})$. Then the image of the map $\kappa:X(\{\mu\},b)\rightarrow \pi_1(G)_{\Gamma_0}$ equals $c_{b,\{\mu\}}\pi_1(G)_{\Gamma_0}^\sigma$. 
\end{lemma}

\begin{remark}
Recall that we have proved in Theorem \ref{KR-He} that $X(\{\mu\}, b) \neq \emptyset$ if and only if $[b] \in B(G, \{\mu\})$. 
\end{remark}

\begin{proof}
For $g \brI \in X(\{\mu\},b)$, we have $\tilde{\k}(g)-\sigma(\tilde{\kappa}(g))=[\mu]-\tilde{\kappa}(b)=c_{b,\{\mu\}}-\s(c_{b,\{\mu\}})$
 thus $\tilde{\k}(X(\{\mu\},b)) \subset c_{b,\{\mu\}}\pi_1(G)_{\Gamma_0}^\sigma$. 

On the other hand, for $\g \in \Omega$ with $\s(\g)=\g$ and $g \brI \in X_w(b)$, we have $g \dot \g \brI \in X_{\g \i w \g}(b)$. Since $\Adm(\{\mu\})$ is stable under the conjugation action of $\Omega$, the right multiplication of $\Omega^\s$ on $\mathcal{FL}$ stabilizes $X(\{\mu\}, b)$. Since $X(\{\mu\}, b) \neq \emptyset$, we have $\k(X(\{\mu\},b))=c_{b,\{\mu\}}\pi_1(G)_{\Gamma_0}^\sigma$. 
\end{proof}

\subsection{}\label{6.1} Recall that the set $B(G, \{\mu\})$ contains a unique minimal element, the basic $\s$-conjugacy class $[\dot \t_{\{\mu\}}]$. In this section we given a description of the connected components of $X(\{\mu\},\dot \t_{\{\mu\}})$. We show that, except the trivial cases, the first obstruction (given in Lemma \ref{1-obstr}) is the only obstruction to connect points in $X(\{\mu\}, \dot \t_{\{\mu\}})$.  In this case $\tilde{\k}(\dot{\tau}_{\mu})=[\mu]\in\pi_1(G)_{\G_0}$ so we may take $c_{\dot{\tau}_{\{\mu\}},\{\mu\}} =1$.

We may write $G_{\text{ad}}$ as $G_{\text{ad}}=G_1\times\cdots\times G_n$, where $G_i$ is are the simple factors of $G_{\text{ad}}$, in other words, the action of $\s$ on the affine Dynkin diagram of $G_i$ is transitive for each $i$. Let $\mu_{\text{ad}}$ be the composition of $\mu$ with $G\rightarrow G_{\text{ad}}$, and $\mu_i$ be the projection of $\mu_{\text{ad}}$ to the factor $G_i$. Then it is easy to see that 
$$X(\{\mu\}_{\text{ad}},\dot \t_{\{\mu\}, \text{ad}})\cong X(\{\mu_1\}, \dot \t_{\{\mu_1\}})\times \cdots\times X(\{\mu_n\}, \dot \t_{\{\mu_n\}}),$$
where $\mu_i$ is the projection of $\mu_{\text{ad}}$ to the factor $G_i$. 

We say that $\mu$ is {\it essentially noncentral} if for each $i$, $\mu_i$ is not central in $G_i$. 

\begin{theorem}\label{conn} Suppose that $\mu$ is essentially noncentral. Then the Kottwitz homomorphism induces an isomorphism  $$\pi_0(X(\{\mu\}, \dot \t_{\{\mu\}}))\cong \pi_1(G)_{\Gamma_0}^\sigma.$$
\end{theorem}

\begin{remark}
If $\mu$ is central, then $\Adm(\{\mu\})=\{t^{\underline\mu}\}$ and $$X(\{\mu\}, \dot \t_{\{\mu\}})=X_{t^\mu}(\dot \t_{\underline\{\mu\}})=G(F)/\brI(F).$$ 
\end{remark}

We begin with several results concerning the structure of the admissible set $\mbox{Adm}(\{\mu\})$. 

\begin{lemma}\label{Tom}
Let $\CO$ be a(n ordinary) straight conjugacy class of $\brW$. If $\CO \cap \Adm(\{\mu\}) \neq \emptyset$, then all the straight elements in $\CO$ are contained in $\Adm(\{\mu\})$.
\end{lemma}

\begin{proof}
Let $w \in \CO \cap \Adm(\{\mu\})$. By \cite[Corollary 2.6]{He00}, there exists a straight element $w'$ in $\CO$ with $w' \le w$. Since $\Adm(\{\mu\})$ is closed under Bruhat order, we have $w' \in \Adm(\{\mu\})$. In other words, $w' \le t^{x(\underline \mu)}$ for some $x \in \brW_0$. 

Let $w''$ be another straight element in $\CO$. By Theorem \ref{HN-min} (2), $w'' \approx w'$. The statement follows from \cite[Lemma 4.5]{Tom}.
\end{proof}

\begin{lemma}\label{s-tau}
Suppose that $G_{\brF, \text{ad}}$ is simple and $\mu$ is not central. Then $s_j \t_{\{\mu\}} \in \Adm(\{\mu\})$ for all $j \in \brbS$.
\end{lemma}

\begin{remark}
Note that $G_{\brF, \text{ad}}$ is simple if and only if the corresponding affine Dynkin diagram is connected. The group $G_{\text{ad}}$ is simple if and only if the action $\s$ on the set of connected components of affine Dynkin diagram is transitive. 
\end{remark}

\begin{proof}
Set $w=t^{\underline \mu} \t_{\{\mu\}} \i \in \brW_a$. Let $K$ be the minimal $\Ad(\t_{\{\mu\}})$-stable subset of $\brbS$ that contains $\text{Supp}(w)$. If $K \neq \brbS$, then $W_K$ is a finite Coxeter group since $G_{\text{ad}}$ is simple. In this case, let $n \in \BN$ such that $n$ is divisible by $|\text{Aut}(W_K)|$ and by the order of $\Ad(\t_{\{\mu\}}) \in \text{Aut}(\brW)$, we have that $t^{n \underline \mu}=(w \t_{\{\mu\}})^n=\t_{\{\mu\}}^n$ is a central element in $\brW$. That is a contradiction.

Hence $K=\brbS$. In particular, there exists $j' \in \text{Supp}(w)$ and $i \in \BN$ such that $\Ad(\t_{\{\mu\}})^i s_{j'}=s_j$. We have $$s_{j'} \t_{\{\mu\}} \approx \t_{\{\mu\}} s_{j'}=\Ad(\t_{\{\mu\}})(s_{j'}) \t_{\{\mu\}} \approx \cdots \approx s_j \t_{\{\mu\}}.$$

Since $s_{j'} \t_{\{\mu\}} \le w \t_{\{\mu\}}$, we have $s_{j'} \t_{\{\mu\}} \ in \Adm(\{\mu\})$. By Lemma \ref{Tom}, $s_j \t_{\{\mu\}} \in \Adm(\{\mu\})$. 
\end{proof}

We will use this lemma to construct curves in $X(\{\mu\},b)$ connecting the elements $u_{-J}$ of section 5.4 to the identity.

\begin{proposition}\label{uJ}Let $\t=\Ad(\dot \t_{\{\mu\}}) \circ \s$. Let $J\in\tilde{\mathbb{S}}(\tau)$ and $u_{-J}$ be the corresponding element of $J_b$ constructed in 5.4. Then $u_{-J} \brI$ and $\brI$ lie in the same connected component of $X(\{\mu\}, \dot \t_{\{\mu\}})$.
\end{proposition}

\begin{proof} 
Let $\a$ be an affine root. Then $\mathcal{U}_\alpha$ and $\mathcal{U}_{\alpha+}$ are group schemes associated to finite free modules over $W(\mathbf{k})$ and the quotient $\mathcal{U}_\alpha(\mathcal{O}_{\brF})/\mathcal{U}_{\alpha+}(\mathcal{O}_{\brF})$ is a $1$-dimensional $\mathbf{k}$ vector space. Let $W(\mathbf{k})\rightarrow \mathcal{U}_\alpha(\mathcal{O}_{\brF})$ be any map lifting $\mathbf{k}\cong \mathcal{U}_\alpha(\mathcal{O}_{\brF})/\mathcal{U}_{\alpha+}(\mathcal{O}_{\brF})$. This induces a morphism of $\mathbf{k}$-schemes

$$x_\alpha:\mathbb{A}_{\mathbf{k}}^{1,p^{-\infty}}\rightarrow L^+\mathcal{U}_\alpha,$$ where $\mathbb{A}_{\mathbf{k}}^{1,p^{-\infty}}$ is the perfection of $\mathbb{A}^1/\mathbf{k}$. 

By Lemma \ref{s-tau} and the remark below, $\t$ acts transitively on the set of connected components of the affine Dynkin diagram of $\brbS$ and there exists $s \in J$ such that $s \t_{\{\mu\}} \in \Adm(\{\mu\})$. Let $\breve \CK$ be the standard parahoric subgroup of $G(\brF)$ associated to $J$. Then $\breve \CK/\brI$ is the finite flag variety of the reductive quotient of $\breve \CK$. Set $$Y=\{g \brI; g \i \dot \t_{\{\mu\}} \s(g) \in \brI \dot \t_{\{\mu\}} \cup \brI \dot s \dot \t_{\{\mu\}} \brI=\{g \brI; g \i \t(g) \in \brI \cup \brI \dot s \brI\} \subset \breve \CK/\brI.$$

Since $s \t_{\{\mu\}} \in \Adm(\{\mu\})$, $Y \subset X(\{\mu\}, \dot \t_{\{\mu\}})$. By definition, $u_{-J} \brI$ and $\brI$ are contained in $Y$. By \cite[Theorem 1.1]{Go}, $Y$ is connected. It is a curve in $X(\{\mu\}, \dot \t_{\{\mu\}})$ that connects $u_{-J} \brI$ with $\brI$. 
\end{proof}

\subsection{Proof of Theorem \ref{conn}}
By \S\ref{6.1}, it suffices to consider the case where $G$ is adjoint and simple. It is easy to see that the only $\s$-straight element in $\brW$ that corresponds to $\dot \t_{\{\mu\}}$ is $\t_{\{\mu\}}$. By Theorem \ref{theorem2}, every point of $X(\{\mu\}, \dot \t_{\{\mu\}})$ is connected to a point in $X_{\t_{\{\mu\}}}(\dot \t_{\{\mu\}})$.

By Theorem \ref{jb-X}, $J_b$ acts transitively on $X_{\t_{\{\mu\}}}(\dot \t_{\{\mu\}})$. For any $J \in \brbS(\t)$, $u_{-J}$ is contained in the parahoric subgroup of $G(\brF)$ corresponding to $J$ and thus is contained in the kernel of the Kottwitz map $\tilde{\kappa}$. By Theorem \ref{Jb}, $\tilde{\kappa}: J_{\dot \t_{\{\mu\}}}\rightarrow \pi_1(G)_{\Gamma_0}^\sigma$ is surjective and the kernel is generated by $u_{-J}$ for $J \in \brbS(\t)$ and $\brI^\t$. It remains to show that for any $j \in \ker(\k) \cap J_{\dot \t_{\{\mu\}}}$, $j \brI$ and $\brI$ are in the same connected component of $X(\{\mu\}, \dot \t_{\{\mu\}})$. 

We have $j=j_1 \cdots j_n$ where $j_i\in \brI^\t$ or $j_i=u_{-J}$ for some $J\in \brbS(\tau)$. Set $j'=j_1 \cdots j_{n-1}$. If $j_n \in \brI^\t$, then $j \brI=j' \brI$. If $j_n=u_{-J}$, then by Proposition \ref{uJ},  $j_n \brI$ and $\brI$ lie in the same connected component of $X(\{\mu\},\dot \t_{\{\mu\}})$, and hence $j \brI$ and $j' \brI$ are in the same connected component of $X(\{\mu\},\dot \t_{\{\mu\}})$. The statement follows from induction on $n$.

\section{Second reduction theorem}\label{8}

\subsection{}\label{s-trivial} In this section, we assume that $G$ is residually split, i.e. $\s$ acts trivially on $\brW$. Note that this condition is more general than $G$ being split, however the assumption does imply that $G$ is quasi-split.

Let $[b] \in B(G, \{\mu\})$ and $\CO_{[b]}$ be the unique straight $\sigma$-conjugacy class associated to it in the sense of Theorem \ref{str-bg}. By \cite[Proposition 4.1]{He00}, $\Adm(\{\mu\})$ contains a straight element in $\CO_{[b]}$. By Lemma \ref{Tom}, 

(a) Any straight element in $\CO_{[b]}$ is contained in $\Adm(\{\mu\})$. 

Let $w$ be a straight element with $\dot w \in [b]$. Then by (a), $w \in \Adm(\{\mu\})$. By definition, $w=t^{\l_w} x$ for some $x \in \brW_0$. Let $\overline{\lambda}_w$ denote the dominant representative of $\l_w$, then $t^{\bar \l_w} \in \brW_0 \Adm(\{\mu\}) \brW_0=\Adm(\{\mu\})^{K_0}$. Note that $t^{\bar \l_w} \in {}^{K_0} \brW$. By \cite[Proposition 6.1]{He00}, $t^{\bar \l_w} \in \Adm(\{\mu\})$. Hence 

(b) For any straight element $w \in \brW$ with $\dot w \in [b]$, we have $t^{\lambda_w}\in\Adm(\{\mu\})$. 

Since $\s$ acts trivially on $\brW$, by Theorem \ref{fundamental}, the linear part of $w$ fixes $\nu_w$. Therefore

(c) For any straight element $w$, $\dot w$ is basic in $M_{\nu_w}$.

\subsection{} Recall we have fixed a special vertex $\mathfrak{s}$ lying in the closure of the base alcove $\mathfrak{a}$. This determines the semi-direct product decomposition $\brW\cong X_*(T)_{\Gamma_0}\ltimes \brW_0$. For $w\in\brW$, let $\lambda_w$ denote its $X_*(T)_{\Gamma_0}$-part and let $\l_{w,}\in X_*(T)_{\G_0}$ denote its $M_{\nu_w}$-dominant conjugate which lies $\Adm(\{\mu\})$ by Lemma \ref{Tom}. We let $\tilde{\l}_{w,dom}\in X_*(T)$ denote an $M_{\nu_w}$-dominant lift of $\l_{w,dom}$. Let $M_{\nu_w}$ denote the semistandard Levi subgroup defined in \S\ref{Mv} and $\brI_{M_{\nu_w}}=M_{\nu_w}(\brF)\cap\brI$ be the associated Iwahori subgroup of $M_{\nu_w}$. Let $\{\tilde \l_w\}_{M_{\nu_w}}$ be the $M_{\nu_w}$-conjugacy class of characters containing $\tilde \l_w$. 

We write $M_{[b]}$ for $M_{\bar \nu_b}$. This is a standard Levi subgroup of $G$. We say that $[b]$ is {\it essentially nontrivial} in $M_{[b]}$ if for some (or equivalently, any) straight element $w \in \brW$ with $\dot w \in [b]$, $\tilde \l_w$ is essentially noncentral in $M_{\nu_w}$. 

\subsection{}\label{adm-levi}
Recall that we denote the Bruhat order on $\brW$ by $\le$. Let $\brW_{\nu_w}$ be the Iwahori-Weyl group of $M_{\nu_w}$ and $\le_{\nu_w}$ be the Bruhat order on $M_{\nu_w}$. By definition, for any $x \in \brW_{\nu_w}$ and an affine reflection $r$ of $\brW_{\nu_w}$, we have $x \le_{\nu_w} r x$ if and only if $x \le r x$. Therefore, we have 

(a) Let $x, y \in \brW_{\nu_w}$. If $x \le_{\nu_w} y$, then $x \le y$. 

In particular,

(b) $\Adm^{M_{\nu_w}}(\{\tilde \l_w\}_{M_{\nu_w}}) \subset \Adm(\{\mu\})$. 

The main result of this section is the following

\begin{theorem}\label{nonbasic}
Assume that $G$ is residually split. Then for any $[b] \in B(G, \{\mu\})$, we have a natural morphism 
\[\bigsqcup_{w \in \brW \text{ is a straight element with } \dot w \in [b]} X^{M_{\nu_w}}(\{\tilde \l_w\}_{M_{\nu_w}}, \dot w) \to X(\{\mu\}, b),\] which induces a surjection \[\tag{a}\coprod_{w \in \brW \text{ is a straight element with } \dot w \in [b]} \pi_0(X^{M_{\nu_w}}(\{\tilde \l_w\}_{M_{\nu_w}}, \dot w)) \twoheadrightarrow \pi_0(X(\{\mu\}, b)).\]

If moreover, $[b]$ is essentially nontrivial in the Levi subgroup $M_{[b]}$, then the natural morphism induces a surjection $$\coprod_{w \in \brW \text{ is a straight element with } \dot w \in [b]} \pi_1(M_{v_w})_{\Gamma_0}^\s \twoheadrightarrow \pi_0(X(\{\mu\}, b)).$$ 
\end{theorem}

\begin{remark}
In general the situation can be much worse. For example, if $\mathcal{O}$ was the conjugacy class of $t_{\underline{\mu}}$, i.e. $\mathcal{O}$ consists of maximal translation elements in $\Adm(\{\mu\})$, then $X(\{\mu\},b)$ is discrete and we have a bijection $$\coprod_{w\in \mathcal{O}, w \ \text{straight}}M_{v_w}(F)/\brI_M^\sigma\cong X(\{\mu\},b).$$ 
\end{remark}

\begin{proof} Let $[b] \in B(G, \{\mu\})$ and $w \in \brW$ is a straight element with $\dot w \in [b]$.  Then by \S\ref{adm-levi} (b), we have $$X^{M_{\nu_w}}(\{\tilde \l_w\}_{M_{\nu_w}}, \dot w) \subset X(\{\mu\}, \dot w) \cong X(\{\mu\}, b).$$ Here the second map is given by $g \brI \mapsto h_w g \brI$, where $h_w$ is an element in $G(\brF)$ with $h_w \i \dot w \s(h_w)=b$. 

This defines a morphism $$\bigsqcup_{w \in \brW \text{ is a straight element with } \dot w \in [b]} X^{M_{\nu_w}}(\{\tilde \l_w\}_{M_{\nu_w}}, \dot w) \to X(\{\mu\}, b)$$ and a map $$\coprod_{w \in \brW \text{ is a straight element with } \dot w \in [b]} \pi_0(X^{M_{\nu_w}}(\{\tilde \l_w\}_{M_{\nu_w}}, \dot w)) \to \pi_0(X(\{\mu\}, b)).$$ Both maps depend on the choice of $h_w$. However, since $J_{\dot w} \subset M_{\nu_w}$, the image of the second map in $\pi_0(X(\{\mu\}, b))$ does not depend on the choice of $h_w$. 

By Corollary \ref{cor1}, every element in $X(\{\mu\}, b)$ is connected to an element in $X_w(b)$ for a straight element $w$ with $\dot w \in [b]$. By \cite[Theorem 2.1.4]{GHKR}, $X_w(\dot w) \cong X^{M_{\nu_w}}(\dot w) \subset X^{M_{\nu_w}}(\{\tilde \l_w\}_{M_{\nu_w}}, \dot w)$. This proves the surjectivity of (a).

The ``moreover'' part follows from the surjectivity of (a) and Theorem \ref{conn}. 
\end{proof}

\section{Passing from Iwahori level to parahoric level}
In this section, we study the connected components of $X(\{\mu\}, b)_K$ for a parahoric subgroup $\brK$. We first describe the basic case. 

\begin{theorem}
Suppose that $\mu$ is essentially nontrivial. Then the Kottwitz homomorphism induces an isomorphism $$\pi_0(X(\{\mu\}, \dot \t_{\{\mu\}})_K) \cong \pi_1(G)^\s_{\G_0}.$$
\end{theorem}

\begin{proof}
We have a commutative diagram 
\[
\xymatrix{X(\{\mu\}, \dot \t_{\{\mu\}}) \ar[r]  \ar@{^{(}->}[d] & X(\{\mu\}, \dot \t_{\{\mu\}})_K  \ar@{^{(}->}[d] \\ \mathcal{FL} \ar[r] & Gr_{\brK}}.
\]

By Theorem \ref{KR-He} (2), the map $X(\{\mu\}, \dot \t_{\{\mu\}}) \to X(\{\mu\}, \dot \t_{\{\mu\}})_K$ is surjective. Hence the induced map $\pi_0(X(\{\mu\}, \dot \t_{\{\mu\}})) \to \pi_0(X(\{\mu\}, \dot \t_{\{\mu\}})_K)$ is also surjective. 

We have $\pi_0(\mathcal{FL}) \cong \pi_0(Gr_{\brK}) \cong \pi_1(G)_{\G_0}$. By Theorem \ref{conn}, the map $$\pi_0(X(\{\mu\}, \dot \t_{\{\mu\}})) \to \pi_0(\mathcal{FL}) \cong \pi_0(Gr_{\brK})$$ is injective. By the commutativity of the above diagram, the map $$\pi_0(X(\{\mu\}, \dot \t_{\{\mu\}})) \to \pi_0(X(\{\mu\}, \dot \t_{\{\mu\}})_K)$$ is also injective. Hence $\pi_0(X(\{\mu\}, \dot \t_{\{\mu\}})_K) \cong \pi_1(G)^\s_{\G_0}$.
\end{proof}

We have the following result in the general case. 

\begin{proposition}
Every point in $X(\{\mu\},b)_K$ is connected to a point in $X_{K, x}(b)$ for some $\s$-straight element $x\in \Adm(\{\mu\}) \cap {}^K \brW$.
\end{proposition}

\begin{proof}
Let \begin{align*} X(\{\mu\}, b)^K &=\{g \in G(\brF)/\brI; g \i b \s(g) \in \cup_{w \in \Adm(\{\mu\})_K} \brK \dot w \brK\} \\ &=\sqcup_{w \in \Adm(\{\mu\})^K} X_w(b) \subset \mathcal{FL}.\end{align*} Then the projection map $\mathcal{FL} \to Gr_{\brK}$ induces a surjective map $X(\{\mu\}, b)^K \to X(\{\mu\}, b)_K$ and each fiber is isomorphic to $\brK/\brI$. Therefore 

(a) \quad $\pi_0(X(\{\mu\}, b)^K) \cong \pi_0(X(\{\mu\}, b)_K).$

Note that $\Adm(\{\mu\})^K$ is closed under the Bruhat order. By Theorem \ref{theorem2}, every point in $X(\{\mu\},b)^K$ is connected to a point in $X_w(b)$ for a $\s$-straight element $w$ in $\Adm(\{\mu\})^K$. By \cite[Theorem 6.17]{HR}, there exists a $\s$-straight element $x \in \Adm(\{\mu\}) \cap {}^K \brW$ that is $\s$-conjugate to $w$ by an element in $W_K$. By \cite[Proposition 6.6]{HR}, we have 

(b) \quad $\brK_\s (\brI \dot w \brI)=\brK_\s (\brI \dot x \brI),$ where $\brK_\s \subset \brK \times \brK$ is the graph of the Frobenius map $\s$. 

Note that $\cup_{w \in \Adm(\{\mu\})_K} \brK \dot w \brK$ is stable under the action of $\brK \times \brK$ and hence is also stable under the action of $\brK_\s$. Thus $X(\{\mu\}, b)^K$ is stable under the left action of $\brK$. Since $\brK$ is connected, by (b), every point in $X_w(b)$ is connected in $X(\{\mu\}, b)^K$ to a point in $X_x(b)$. Therefore every point in $X(\{\mu\},b)^K$ is connected to a point in $X_x(b)$ for a $\s$-straight element $x$ in $\Adm(\{\mu\}) \cap {}^K \brW$. Now the proposition follows from (a). 
\end{proof}

Similar to the proof in Theorem \ref{nonbasic}, we have

\begin{theorem}\label{nonbasic2}
Assume that $G$ is residually split. Then for any $[b] \in B(G, \{\mu\})$, we have a natural morphism \[\bigsqcup_{w \in {}^K\brW \text{ is a straight element with } \dot w \in [b]} X^{M_{\nu_w}}(\{\tilde \l_w\}_{M_{\nu_w}}, \dot w) \to X(\{\mu\}, b)_K,\] which induces a surjection \[\coprod_{w \in {}^K \brW \text{ is a straight element with } \dot w \in [b]} \pi_0(X^{M_{\nu_w}}(\{\tilde \l_w\}_{M_{\nu_w}}, \dot w)) \twoheadrightarrow \pi_0(X(\{\mu\}, b))_K.\]

If moreover, $[b]$ is essentially nontrivial in the Levi subgroup $M_{[b]}$, then the natural morphism induces a surjection $$\coprod_{w \in {}^K \brW \text{ is a straight element with } \dot w \in [b]} \pi_1(M_{v_w})_{\Gamma_0}^\s \twoheadrightarrow \pi_0(X(\{\mu\}, b)_K).$$ 
\end{theorem}

Compared to the Iwahori case, the number of straight elements one needs to consider here could be smaller. For example, if $G$ is residually split, $\brK$ is a special maximal parahoric subgroup and $b$ is a translation element, then there is only one straight element involved in the left hand of the statement. 
\section{Verification of Axioms in \cite{HR} for PEL Shimura varieties}\label{10}
In this section we apply the previous construction to study the connected components in the basic locus of some PEL Shimura varieties. 

\begin{theorem}
The five Axioms in \cite{HR} hold for PEL type Shimura varieties associated to unramified groups of type A and C and to odd ramified unitary groups. 
\end{theorem}

The Siegel case is already verified in \cite{HR}. In the rest of this section, we focus on odd ramified unitary groups. The same proofs go through for unramified groups as well. In fact, in the unramified cases, the naive local model is flat and hence there is no need to take closures in the construction of the integral models. Therefore one does not need the coherence conjecture in the verification of Axiom (2). 

As to other PEL type, we expect that the axioms in \cite{HR} still hold. However, there are several technical difficulties to overcome. First, the argument we have uses Dieudonne theory, which requires the moduli description of integral models. Second, for even ramified unitary groups and even orthogonal groups, the stabilizer of a chain of lattice might not be connected and thus does not equal to a parahoric subgroup. This gives extra difficulties in understanding the integral models. We do not investigate these cases here.

\subsection{} Let $E$ be a quadratic imaginary field and $\epsilon: E\rightarrow \mathbb{C}$ be a fixed embedding. Let $W$ be an $n$-dimensional $E$-vector space equipped with a Hermitian form $\phi:W\times W\rightarrow E$, we assume $n=2m+1$ is odd. Let $G$ be the unitary similitude group $$GU(\phi)=\{g\in GL_n(E):\phi(gv,gw)=\phi(v,w),\ \forall v,w\in W \}.$$
Then $G$ arises as the $\mathbb{Q}$ points of a reductive group over $\mathbb{Q}$, also denoted $G$. 

Let $r,s$ be integers such that $r+s=n$ with $s\leq r$, then as in \cite{PR1} there is a homomorphism $h_0:\mathbb{S}\rightarrow G_{\mathbb{R}}$ which induces a complex structure on $W\otimes_{\mathbb{Q}}\mathbb{R}$ such that $\text{Tr}(a|W\otimes_{\mathbb{Q}}\mathbb{R})=s\epsilon(a)+r\overline{\epsilon}(a)$ for $a\in E$ with respect to this complex structure. Let $X$ denote the $G(\mathbb{R})$ conjugacy class of $h_0$, we obtain a Shimura datum $(G,X)$ with reflex field $E$. 

Let $\alpha\in E$ be a totally imaginary element, i.e. $\overline{\alpha}=-\alpha$. Then we have an alternating $\mathbb{Q}$-linear form $\psi:W\otimes_{\mathbb{Q}}W\rightarrow \mathbb{Q}$ given by $$\psi(u,v)=\text{Tr}_{E/\mathbb{Q}}(\alpha^{-1}\phi(u,v)).$$
Upon replacing $\alpha$ by $-\alpha$, we may assume the $\mathbb{R}$ bilinear form on $W\otimes_{\mathbb{Q}}(\mathbb{R})$ given by $\psi(u,h_0(i)v)$ is positive definite. We thus obtain an embedding of Shimura data $\iota: (G,X)\hookrightarrow (GSp(W,\psi),S^\pm)$.

Now let $p>2$ be a prime which ramifies in $E$ and assume $\phi$ is split over $F=\mathbb{Q}_p$. Let $F'=E\otimes_{\mathbb{Q}}\mathbb{Q}_p$, then $F'$ is a ramified quadratic extension of $\mathbb{Q}_p$ and we let $\pi\in\mathcal{O}_{F'}$ be a uniformizer with $\overline{\pi}=-\pi$. Let $e_1,\cdots,e_n$ be a basis for $W\otimes_{E}F'$ such that $\phi(e_i,e_{n+1-j})=\delta_{ij}$, we have the maximal torus $T$ given by the elements $$\text{diag}(a_1,\cdots,a_m,a,a\overline{a}\overline{a}_m^{-1},\cdots,a\overline{a}\overline{a}_1^{-1}).$$
We have the associated Iwahori Weyl group $\brW$ and the affine Weyl group $$\brW_a\cong \mathbb{Z}^m\rtimes \brW_0,$$ where $W_0=S_m\times\{\pm1\}$, see \cite[\S2]{PR1} for these calculations.

For $j\in\{0,\cdots,m\}$, we define lattice $\Lambda_j$ by $$\L_j=\text{span}_{\mathcal O_E} \<\pi^{-1}e_1,\cdots,\pi^{-1}e_j,e_{j+1},\cdots,e_n\>.$$ We may extend $\Lambda_j, j\in J$ to a periodic self-dual lattice chain indexed by $\{kn+j| k\in\mathbb{Z},j\in J\}$ as in \cite[\S 1.d]{PR1}. For $J\subset \{0,\cdots,m\}$, we associate the common stabilizer $K_J$ of the lattices $\{\Lambda_j\mid j\in J\}$ in $G(\BQ_p)$. Then $K_J$ is a parahoric subgroup. We write $\mathcal{G}_J$ for the associated group scheme. The parahoric subgroup $K_J$ is special maximal when $J=\{m\}$ and is Iwahori when $J=\{0,\cdots,m\}$, we write $I$ for this Iwahori subgroup. 

We fix $\mathrm{K^p}\subset G(\mathbb{A}_f^p)$ a sufficiently small compact open subgroup. The we have the associated Shimura variety $Sh_{K_J\mathrm{K}^p}(G,X)$ over $E$. Let $v$ be the place of $E$ above $p$, the naive integral model  $\mathscr{S}^{naive}_{K_J}$ over $\mathcal{O}_{E_v}$ is given by the following moduli problem. Let $S$ be an $\mathcal{O}_{E_v}$ scheme, an $S$ point of $\mathscr{S}^{naive}_{K_J}$ is given by the set of isomorphism classes of tripes $(A,\lambda,\eta)$ where:
\begin{itemize}
\item $A=\{A_j\}$, $j\in\{kn+j| k\in\mathbb{Z},j\in J\}$ is an $\mathcal{L}$ set of abelian varieties over $S$ in the sense of \cite[Definition 6.5]{RZ}.  In particular each $A_j$ is equipped with a map $i:\mathcal{O}_E\otimes\mathbb{Z}_{(p)}\rightarrow \text{End}(A)\otimes\mathbb{Z}_{(p)}$.

\item $\lambda$ is a $\mathbb{Q}$-homogeneous polarization of $A$, \cite[Definition 6.6]{RZ}.

\item $\eta$ is a $\mathrm{K}^p$ level structure, i.e. an isomorphism:
$$ \eta:H_1(A,\mathbb{A}_f^p)\cong W\otimes_{\mathbb{Q}}\mathbb{A}^f_p\mod \mathrm{K}^p$$
respecting the bilinear forms on either side up to a scalar. 
\end{itemize}

Moreover $A$ is required to satisfy the determinant condition of \cite[Definition 6.9]{RZ}, i.e. we have $\text{det}_{\mathcal{O}_S}(b;\Lie(A_j)=\text{det}_{E}(b;V_0)$,
see loc. cit. for details.
The integral model $\mathscr{S}_{K_J}$ is defined to be the closure of the generic fiber in $\mathscr{S}_{K_J}^{naive}.$ 

Now if we base change the embedding $\iota:G\rightarrow GSp(W,\psi)$ to $\mathbb{Q}_p$, we obtain a representation of $G$ satisfying the conditions of \cite[1.2.15]{KP}. Let $\mathcal{GL}$ denote the stabilizer in $GL_{2n}(\mathbb{Q}_p)$ of the lattice chain $\Lambda_i$ considered as a chain of $\mathbb{Z}_p$ lattices in $W\otimes_{\mathbb{Q}}\mathbb{Q}_p$ considered as $\mathbb{Q}_p$ vector space. Then $\mathcal{GL}$ is a parahoric subgroup of $GL_{2n}(\mathbb{Q}_p)$ and we have a closed embedding of group schemes $\mathcal{G}_J\rightarrow \mathcal{GL}$, see \cite[Proposition 3.3]{KP}.

\subsection{} We now verify the axioms of \cite{HR} for these models.

\textbf{Axiom (1)} {\it Compatibility with change in parahoric:} As in \cite[\S 7]{HR}, it suffices to prove the desired properties for the morphism $\pi_{K_J, K_{J'}}$, when $J$ arises from  $J'$  by adding a single element $j$. Let $j'\in J'$ be maximal with $j'<j$. 

Using the description of the naive integral models as a moduli space, one sees that there exists a map $\pi_{J,J'}^{naive}:\mathscr{S}^{naive}_{K_J}\rightarrow\mathscr{S}^{naive}_{K_{J'}}$. This induces a maps of integral models $\pi_{J,J'}:\mathscr{S}_{K_J}\rightarrow \mathscr{S}_{K_{J'}}$.

For an $S$ point $(A,\lambda,\eta)$ of $\mathscr{S}_{K_{J'}}^{naive}$, to give a point in the preimage of $\pi_{J,J'}$ is to give a subgroup of $A_{j'}[p]$ satisfying certain properties, hence this is representable by a closed subscheme of a Hilbert scheme. Thus $\pi_{J,J'}^{naive}$ is proper and hence so is $\pi_{J,J'}$. Since $\pi_{J,J'}$ is surjective on the generic fiber, it has dense image and thus $\pi_{J, J'}$ is surjective.

\smallskip

\textbf{Axiom (2)} {\it Local model diagram:} As in \cite[\S 1.e.4]{PR} that there is a diagram 
$$\mathscr{S}^{naive}_{\mathrm{K_J}}\xleftarrow \pi \tilde{\mathscr{S}}^{naive}_{K_J}\xrightarrow q M_{K_J}^{naive},$$
where $\pi$ is a $K_J$-torsor and $q$ is smooth of relative dimension $\dim G$. Here $M_{K_J}^{naive}$ is the naive local model defined in \cite{RZ}. We thus obtain a diagram
$$\mathscr{S}_{\mathrm{K_J}}\xleftarrow \pi \tilde{\mathscr{S}}_{K_J}\xrightarrow q M^{loc}_{K_J},$$ 
where $M^{loc}_{K_J}$ is the closure of the generic fiber of $M_{K_J}^{naive}$. Since the coherence conjecture has been proved in \cite{Zhu1}, the special fiber of $M^{loc}_{K_J}$ has a stratification by $\Adm(\{\mu\})_{K_J}$, see \cite[\S 4.1]{PR1}. 
We have a Cartesian diagram of morphisms of stacks:
\[
\xymatrix{\mathscr{S}_{K_J}\ar[r]  \ar[d] & [M^{loc}_{{K_J}}/K_J]  \ar[d] \\ \mathscr{S}^{naive}_{K_J} \ar[r] & [M^{naive}_{{K_J}}/K_J].}
\]

We obtain a map $$\lambda_{K_J}: \mathscr{S}_{K_J}(\mathbf{k})\rightarrow\brW_{K_J}\backslash\brW/\brW_{K_J},$$ giving the Kottwitz-Rapoport stratification. The closure relations on $M^{loc}_{K_J}\otimes \mathbf{k}$ follows from the embedding of $M^{loc}_{K_J}\otimes \mathbf{k}$ into an appropriate affine flag variety, see \cite[3.c]{PR1}.

\smallskip

\textbf{Axiom (3)} {\it Newton stratification:} The verification of this Axiom is the same as in \cite[\S7]{HR} in the Siegel case. 

\smallskip

\textbf{Axiom (4) }{\it Joint stratification:} 

a) \& c) The verifications are again the same as in \cite[\S7]{HR} in the Siegel case. We recall how the map $\Upsilon:\mathscr{S}_{K_J}(\mathbf{k})\rightarrow G(\brF)/\mathcal{G}_J(\mathcal{O}_{\brF})_\sigma$ is defined. For $x\in\mathscr{S}_{K_J}(\mathbf{k})$ corresponding to $(A,\lambda,\eta)$, we let $\mathbb{D}_i$ be the Dieudonn\'e module of $\mathcal{A}_i[p^\infty]$ and let $N$ be the common rational Dieudonn\'e module. The $\mathbb{D}_i$ form an $\mathcal{O}_{K}\otimes {W(\mathbf{k})}$ lattice chain inside $N$. Then by \cite[App. to Chapter 3]{RZ} there is a $K$-linear isomorphism $$N\cong W\otimes_{\mathbb{Q}}\brF$$ taking $\mathbb{D}_i$ to $\Lambda_i$ for all $i\in\mathbb{Z}$ and respecting the forms up to a scalar. Under this isomorphism, the Frobenius is given by $\delta(\text{id}_W\otimes\sigma)$ for some $\delta\in G(\brF)$ which is well defined up to $\sigma$ conjugation by $\mathcal{G}_J(\mathcal{O}_{\brF})$. 

b) By \cite[Lemma 3.11]{HR}, we may assume $J=\{0,\cdots,m\}$. By Axiom (5) below and \cite[Theorem 4.1]{HR}, the image of $\Lambda_I:\mathscr{S}_{I}(\mathbf{k})\rightarrow \Adm(\{\mu\})$ is surjective. Let $b\in l_I^{-1}(w)$ for some $w\in\Adm(\{\mu\})$, then $b\in \brI\dot{w}\brI$. It follows that $1\in X(\{\mu\},b)$ and so $b\in B(G,\{\mu\})$ by \ref{KR-He}. By \cite[Theorem 5.4]{HR}, $\delta_I$ is surjective, hence there exists $x\in\mathscr{S}_I(\mathbf{k})$ such that the $\delta\in G(\brF)$ associated to $x$ as above satisfies $[\delta]=[b]$ in $B(G)$.

Let $g\in G(\brF)$ such that $g^{-1}\delta\sigma(g)=b$. As in the verification Axiom (5) below, $g$ corresponds to a point $gx\in \mathscr{S}_I(\mathbf{k})$, and we have $\Upsilon(gx)=g \delta \sigma(g) \i=b$ in $G(\brF)/\brI_{\sigma}$.

\smallskip

\textbf{Axiom (5)} {\it Basic non-emptiness:} Recall $I\subset G(\mathbb{Q}_p)$ is the Iwahori subgroup corresponding to $\{0,\cdots,m\}$. By \cite{Ki15}, the basic locus is nonempty. Let $x\in\mathscr{S}_{K_J\mathrm{K}^p}(\mathbf{k})$ be a point which lies in the basic locus. This corresponds to a triple $(A_i,\lambda,\eta)$ as above. As in Axiom 4) we obtain a $\delta\in G(\brF)$ where $[\delta]=[b]_{basic}\in B(G)$ the unique basic $\sigma$-conjugacy class. Recall $\tau_{\{\mu\}}$ is the unique element in $\Omega\cap\Adm(\{\mu\})$. A simple calculation shows that $\tau_{\{\mu\}}$ is represented by the element $\dot{\tau}_{\{\mu\}}=\text{diag}(\pi,\cdots,\pi)\in T(\brF)$.

Let $g\in X_{\tau_{\{\mu\}}}(\delta)$ which is non-empty since $\delta$ is basic,
 then  $g^{-1}\delta\sigma(g)\in \brI\dot{\tau}_{\{\mu\}}\brI$. Since $\dot{\tau}_{\{\mu\}}$ is central, we have $\delta\sigma(g)\in g\dot{\tau}_{\{\mu\}}\brI$. Let $\mathbb{D}_j$ denote the Dieudonn\'e module of $\mathscr{G}_j=A_j[p^\infty]$ for $j\in\mathbb{Z}$,  then we have $$pg\mathbb{D}_j\subset\delta\sigma(g)\mathbb{D}_j=\dot{\tau}_{\{\mu\}}g\mathbb{D}_j\subset g\mathbb{D}_j$$
 Thus $g\mathbb{D}_j$ corresponds to an $p$-divisible $g\mathscr{G}_j$ which is isogenous to $A_j[p^\infty]$. Since $g$ is $K$-linear, the action of $\mathcal{O}_K$ on $\mathscr{G}_j$ extends to an action on $g\mathscr{G}_j$. We thus obtain an $\mathcal{L}$-set of abelian varieties $gA_j, j\in \mathbb{Z}$. Since $G\subset GSp(W,\psi)$, the polarization $\lambda$ extends to a polarization $g\lambda$ of the $\mathcal{L}$-set $A_j$. The prime to $p$ level structure $\eta$ extends to $gA_j$, we thus obtain a triple $(gA_j,g\lambda,\eta)$ as above. The determinant condition holds since $(\text{Lie}A_j)^\vee\cong g\mathbb{D}_j/\dot{\tau}_{\{\mu\}}g\mathbb{D}_j$, we thus obtain a $\mathbf{k}$-point of $gx\in \mathscr{S}_I^{naive}(\mathbf{k})$.  It follows from the pullback diagram in the verification of Axiom (2) that $gx\in \mathscr{S}_I(\mathbf{k})$, and by construction we have $\lambda_I(gx)=\tau_{\{\mu\}}$.
 
 To show $$\lambda_I^{-1}(\tau_{\{\mu\}})\rightarrow \pi_0(\mathscr{S}_I\otimes_{\mathcal{O}_{E_v}}\mathbf{k})$$ is surjective, we follow the same strategy as in \cite{HR}. Indeed, in this case a good theory of compactifications exists by \cite{Mad}, see also \cite{Lan}. We thus have an isomorphism $$\pi_0(\mathscr{S}_I\otimes_{\mathcal{O}_{E_v}}\mathbf{k})\cong\pi_0(Sh_{I\mathrm{K}^p}(G,X)\otimes_E\mathbb{C})$$
 
 By \cite{De} $$\pi_0(Sh_{I\mathrm{K}^p}(G,X)\otimes_E\mathbb{C})=G(\mathbb{Q})_+\backslash G(\mathbb{A}^f)/I\mathrm{K}^p$$ where $G(\mathbb{Q})_+$ denotes the elements of $G(\mathbb{Q})$ which map to the connected component of the identity in $G_{ad}(\mathbb{R})$, hence $\varprojlim_{\mathrm{K}^p}\pi_0(\mathscr{S}_{I}\otimes_{\mathcal{O}_{E_v}}\mathbf{k})=G(\mathbb{Q})_+^-\backslash G(\mathbb{A}^f)/I$ where $G(\mathbb{Q})_+^-$ is the closure of $G(\mathbb{Q})_+$ in $G(\mathbb{A}^f)$. By the weak approximation theorem $G(\mathbb{Q})$ is dense in $G(\mathbb{R})\times G(\mathbb{Q}_p)$, hence $G(\mathbb{Q})_+$ is dense if $G(\mathbb{Q}_p)$. It follows that $G(\mathbb{A}^f)$ acts transitively on $\varprojlim_{\mathrm{K}^p}\pi_0(\mathscr{S}_I\otimes_{\mathcal{O}_{E_v}}\mathbf{k})$. Thus for $x\in\lambda_I(\tau_\{\mu\})$ corresponding to $(A,\lambda,\eta)$ we may modify the prime to $p$ level structure to pass to any connected component.

In fact in this case we have the following stronger result. We write $\mathscr{S}_{K_J}^{[b]_{basic}}=\delta_{K_J}^{-1}([b]_{basic})$, a closed subvariety of the special fiber $\mathscr{S}_{K_J}\otimes\mathbf{k}$.

\begin{proposition}\label{10.1} Let $x\in \mathscr{S}_{K_J}^{basic}(\mathbf{k})$. Then there exists a point $x'$ such that $x$ and $x'$ are connected in $\mathscr{S}_{K_J}^{[b]_{basic}}$ and $\lambda_I(x')=\tau_{\{\mu\}}.$
\end{proposition}
 \begin{proof}
 Let $(A,\lambda,\eta)$ denote the triple corresponding to $x$ and $\delta\in G(\brF)$ be the element associated to $x$ in 4). By Theorem \ref{theorem2}, if $Y\subset X(\{\mu\},\delta)$ denotes the connected component containing $1$, we have $Y\cap X_{\tau_{\mu}}(\delta)\neq\emptyset$. By Theorem \ref{CKV}, there exists $g_0\in X_{\tau_{\mu}}(\delta)$ such that $1\sim g_0$ in the sense of $\ref{CKV2}$.
 
 Let $g\in X(\{\mu\},\delta)(\mathscr{R})$,  where $\mathscr{R}$ is a frame for a smooth integral $\mathbf{k}$-algebra $R$, (see Definition \ref{CKV1}). Let $\mathbb{D}_j$ denote the Dieudonn\'e module associated to $\mathscr{G}_j=A_j[p^\infty]$. By \cite{Zhou} (see also \cite[Lemma 1.4.6]{Ki}), upon replacing $\mathscr{R}$ by $\mathscr{R}'_n$, $g\mathbb{D}_j$ corresponds to a chain of $p$-divisible groups $g\mathscr{G}_{j}$ over $R$ with an identification $\mathbb{D}(g\mathscr{G}_{j})(\mathscr{R})\cong g\mathbb{D}_j$. Here, $\mathscr{R}'$ denotes the canonical frame for an \'etale covering $R\rightarrow R'=\mathscr{R}'/p\mathscr{R}'$, and $\mathscr{R}_n'$ denotes the $\mathcal{O}_{\brF}$-algebra with underlying ring $\mathscr{R}'$ and whose structure map is $\sigma^n$. 
 
 The $g\mathscr{G}_j$ correspond to abelian varieties $g{A}_j$ over $R$ together with a quasi isogeny $g{A}_j\rightarrow A_j\otimes_{\mathbf{k}}R$. As before $gA_j$  is equipped with an action of $\mathcal{O}_K$, a polarization $g\lambda$ and  prime to $p$ level structure $\eta$. We thus obtain a triple $(gA,g\lambda,\eta)$ over $R$. To check the determinant condition, we have to check the equality of two polynomials with coefficients over $R$. It thus suffices to check this for all $\mathbf{k}$-points $s:R\rightarrow \mathbf{k}$. 
 
 Let $g(s)\in G(\brF)$ denote the point induced by the unique $\sigma$ equivariant lift $\mathscr{R}\rightarrow \mathcal{O}_{\brF}$ of $s$. The Dieudonn\'e module of $gA_j[p^\infty]$ at $s$ is given by $g(s)\mathbb{D}_j\subset \mathbb{D}_j\otimes\brF$, and the Lie algebra Lie$(g(s)A_j)$ of $gA_j$ at $s$ can be identified with the dual of $\delta\sigma(g(s))\mathbb{D}_j/g(s)\mathbb{D}_j\cong g(s)^{-1}\delta\sigma(g(s))\mathbb{D}_j/\mathbb{D}_j$.
 
 By definition of $g$, we have $g(s)^{-1}\delta\sigma(g(s))\in\brI \dot{w}\brI$ for some $\dot{w}\in\Adm(\{\mu\})$. The isomorphism class of the $\mathcal{O}_E$ module Lie$(g(s)A_j)$ depends only on $\dot{w}$. Write $M_{j,w}$ for $\mathcal{O}_E$-module $\dot{w}(\Lambda_j\otimes_{\mathbb{Z}_p}\mathcal{O}_{\brF})/(\Lambda_j\otimes_{\mathbb{Z}_p}\mathcal{O}_{\brF})$, then we have an isomorphism Lie$(g(s)A_j)^\vee\cong M_{j,w}$. 
 
 Since $\lambda_I$ is surjective, we may take $x'\in \lambda_I^{-1}(w)$ corresponding to a triple $(A',\lambda',\eta')$. Then by definition we have Lie$(A_j')^\vee\cong M_{j,w}$. In particular for any $b\in\mathcal{O}_E$, we have the following equalities of polynomial functions:
 $$\text{det}_{\mathbf{k}}(b;\Lie(g(s)A_j))=\text{det}_{\mathbf{k}}(b;M_{j,w})=\text{det}_{\mathbf{k}}(b;\Lie(A_j'))=\text{det}_{{E}}(b;V_0)$$
 It follows that $\text{det}_R(b,\Lie(gA))=\text{det}_E(b,V_0)$, and so the triple $(gA,\lambda,\eta)$ corresponds to a morphism $$i_R:\Spec R\rightarrow \mathscr{S}_{I}^{naive}$$
 
 For each $s:R\rightarrow \mathbf{k}$ this map factors through $\mathscr{S}_I$, hence $i_R$ factors through $\mathscr{S}_I$ and in fact through $\mathscr{S}_I^{[b]_{basic}}$. Since $1\sim g_0$, we may find a sequence of maps $i_R$ which connects $x$ to a point $x'$ in $\lambda_I^{-1}(\tau_{\{\mu\}})$.
\end{proof}

\appendix

\section{Various definitions of connected components}
 In this section we relate the notions of connected components in this paper to that studied in \cite{CKV}. Although the algebraic structure on the affine Deligne-Lusztig varieties was not known at the time, the authors were still able to define a notion of connected components. That paper deals with the case of unramified groups and hyperspecial level structure, however it is relatively straightforward to generalize their notion of connected components to our context (i.e. parahoric level structure). Although it is much more natural to talk about connected components in the Zariski topology, nevertheless the notion in \cite{CKV} is useful for applications to Shimura varieties and Rapoport-Zink spaces. Therefore it is useful to know that the two notions coincide.
 
We recall some of the definition of \cite{CKV}. Recall that $\mathbf{k}$ is an algebraic closure of $\mathbb{F}_q$. For simplicity we only consider the case of Iwahori level, the arguments however work for any parahoric. Thus let $\mathcal{G}$ denote the group scheme over $\mathcal{O}_F$ associated to the Iwahori $\brI$.

\begin{definition}\label{CKV1}
Let $R$ be  $\mathbf k$ algebra. A frame for $R$ is a $p$ torsion free, $p$-adically complete and separated $\mathcal{O}_{\brF}$ algebra $\mathscr{R}$ equipped with an isomorphism $R\cong \mathscr{R}/p\mathscr{R}$ and a lift (again denoted $\sigma$) of the the Frobenius $\sigma$ on $R$.
 \end{definition}
 
 Let $R$ be as above and fix $\mathscr{R}$ a frame for $R$. We write $\mathscr{R}_{\brF}$ for $\mathscr{R}[\frac{1}{p}]$. If $\kappa$ is any perfect field of characteristic $p$ and $s: R\rightarrow \kappa$ is a map, then there is a unique $\sigma$-equivariant map $\mathscr{R}\rightarrow W(\kappa)$, also denoted $s$. 
 
 Let $g\in G(\mathscr{R}_{\brF})$. For $C\subset\brW$, we write 
 $$S_C(g)=\bigcup_{w\in C}\{s\in\Spec R|s(g^{-1}b\sigma(g))\in\mathcal{G}(W(\mathbf{k}(s)))w\mathcal{G}(W(\mathbf{k}(s)))\}$$
where $\mathbf{k}(s)$ is an algebraic closure of residue field $k(s)$ of $s$. Note that this only depends on the image of $g\in G(\mathscr{R}_{\brF})/\mathcal{G}(\mathscr{R})$, hence we can define $S_C(g)$ for any element of $g\in G(\mathscr{R}_{\brF})/\mathcal{G}(\mathscr{R})$. For $b\in G(\brF)$, we define the set $$X_C(b)(\mathscr{R})=\{g\in G(\mathscr{R}_{\brF})/\mathcal{G}(\mathscr{R})| S_C(g)=\Spec R\}$$
When $C=\Adm(\{\mu\})$ we write $X(\{\mu\},b)(\mathscr{R})$ for $X_C(b)(\mathscr{R})$.
\begin{remark}
This is slightly different to the definition in \cite{CKV}. It is possible to define a version of the mixed characteristic affine flag variety as in \cite{CKV} in our context using $\mathcal{G}$-torsors. However using that $\mathcal{G}$-torsors are \'etale locally trivial, it can be shown that ``\'etale locally" on $\mathscr{R}$, these notions coincide. In particular, the connected components of $X_C(b)$ will be the same.
\end{remark}

For $R=W(\mathbf{k})$, we write $X_C(b)=X_C(b)(W(\mathbf{k}))$, this is compatible with the definitions Section 2.1.

 Note that any $g\in X_C(b)(\mathscr{R})$ defines an $R^{perf}$ point of $X_C(b)$ (considered as a perfect scheme).  Indeed, let $\mathscr{R}^\infty=\lim_{\rightarrow n}\mathscr{R}$, where the transitions maps are given by $\sigma$, thus $\mathscr{R}^\infty$ is a flat $\mathcal{O}_{\brF}$ algebra lifting $R$. Then if we denote by $\widehat{\mathscr{R}^\infty}$ the $p$-adic completion of $\mathscr{R}^\infty$, we have an isomorphism $\widehat{\mathscr{R}^\infty}\cong W(R^{perf})$ since $W(R^{perf})$ is the unique $p$-adically complete flat $\mathbb{Z}_p$ algebra lifting $R^{perf}$ and $\widehat{\mathscr{R}^\infty}$ gives such a lifting. The composition $\mathscr{R}\rightarrow\widehat{\mathscr{R}^\infty}\cong W(R^{perf})$ induces a  point $g\in G(W(R^{perf})[\frac{1}{p}])/\mathcal{G}(W(R^{perf}))$ and hence an element, also denoted $g$, of $\mathcal{FL}(R^{perf})$. The conditions defining $X_C(b)(\mathscr{R})$ imply that the corresponding point in $\mathcal{FL}$ lies in $X_C(b)$, since $R$ and $R^{perf}$ have the same points.

\begin{definition}\label{CKV2} For $g_0,g_1\in X(\{\mu\},b)$ and $R$ a smooth  $\mathbf{k}$-algebra with connected spectrum and frame $\mathscr{R}$, we say $g_0$ is connected to $g_1$ via $R$ if there exists $g\in X(\{\mu\},b)(\mathscr{R})$ and two $\mathbf{k}$-points $s_0,s_1$ of $\Spec R$ such that $s_0(g)=g_0$ and $s_1(g)=g_1$. 

We write $\sim$ for the equivalence relation on $X(\{\mu\},b)$ generated by the relation $g_0\sim g_1$ if $g_0$ is connected to $g_1$ via some $R$ as above, and we write $\pi_0'(X(\{\mu\},b))$ for the set of equivalence classes.
\end{definition}

\begin{theorem}\label{CKV}$g_0\sim g_1$ if and only if $g_0$ is connected to $g_1$ in the Zariski topology. In particular $$\pi_0(X(\{\mu\},b))=\pi_0'(X(\{\mu\},b)).$$
\end{theorem}
\subsection{Proof of  Theorem \ref{CKV}} Clearly, if $g_0\sim g_1$, then $g_0$ and $g_1$ are connected in the Zariski topology.  Indeed wlog. we may assume $g_0$ is connected to $g_1$ via $R$, and the above construction gives a $\Spec R^{perf}$ point in $X(\{\mu\},b)$ which connects the two points. The rest of this section will be devoted to proving the converse.

Suppose $g_0$ and $g_1$ are connected in the Zariski topology. Since $\mathcal{FL}$ is an increasing union of perfections of projective schemes, we may assume $g_0, g_1\in S^{perf}\subset \mathcal{FL}$, where $S^{perf}$ is the perfection of a projective scheme $S$ over $\mathbf{k}$. Considering $g_0$ and $g_1$ as points on $S$, we can find a morphism $C\rightarrow S$, where $C$ is a smooth curve over $\mathbf{k}$, whose image contains $g_0$ and $g_1$, or at least we can connect them up with finitely many such curves. Taking affine open covers of $C$, we may reduce to the case $C=\Spec R$, for a smooth $\mathbf{k}$ algebra $R$, and where $g_0$ and $g_1$ are in the image of $C^{perf}$ in $\mathcal{FL}$.

We need to find a frame $\mathscr{R}$ for $R$, and an element $g\in X(\{\mu\},b)(\mathscr{R})\subset G(\mathscr{R}_{\brF})/\mathcal{G}(\mathscr{R})$. which connects up the two points $g_0$ and $g_1$. By assumption, the curve $C\in X(\{\mu\},b)(R^{perf})\subset\mathcal{FL}(R^{perf})$ connects $g_0$ and $g_1$. Upon passing to an fpqc cover, we may assume $C^{perf}$ comes from a point $g\in G(W(R^{perf})[\frac{1}{p}])$. 

Now let $\mathscr{R}$ be any frame for $R$. Indeed a frame exists since we may choose any $p$-adically complete $W(\mathbf{k})$ algebra $\mathscr{R}$ lifting $R$. Then all obstructions to lifting $\sigma$ lie in positive degree coherent cohomology vanishes. As before we have an isomorphism $\widehat{\mathscr{R}^\infty}\cong W(R^{perf})$. We would like to show $g\in G(\widehat{\mathscr{R}^\infty_{\brF}})/\mathcal{G}(\widehat{\mathscr{R}^\infty})$ arises from an element of $G(\mathscr{R}^\infty_{\brF})/\mathcal{G}(\mathscr{R}^\infty)$. This follows from the following more general Proposition:

\begin{proposition}\label{prop5}
Let $\mathcal{R}$ be a flat $\mathcal{O}_{\brF}$-algebra such that $\mathcal{R}/p\mathcal{R}$ is an integral domain  and let $\mathscr{R}$ be its $p$-adic completion. Suppose every element $r\in\mathcal{R}$ whose reduction mod $p$ is a unit is itself a unit, then for any flat affine algebraic group $\mathcal{G}/\mathcal{O}_{\brF}$ with generic fibre $G$, the natural map $$G(\mathcal{R}_{\brF})/\mathcal{G}(\mathcal{R})\rightarrow G(\mathscr{R}_{\brF})/\mathcal{G}(\mathscr{R})$$ is a bijection. 
\end{proposition}

Granting this proposition we may prove the Theorem. Indeed $\mathscr{R}^\infty$ is flat over $\mathcal{O}_{\brF}$ since it is a direct limit of flat $\mathcal{O}_{\brF}$-algebras. Also, for any $r\in\mathscr{R}^\infty$ which reduces to a unit in $\mathscr{R}^\infty/p\mathscr{R}^\infty$, we have $r\in\mathscr{R}_n$ where $\mathscr{R}_n$ is the ring  $\mathscr{R}$ regarded as an $\mathcal{O}_{\brF}$ algebra via the map $\sigma^n$. Increasing $n$ if necessary, we may assume the image of $r$  in $\mathscr{R}_n/p\mathscr{R}_n$ is a unit, hence since $\mathscr{R}_n$ is $p$-adically complete, $r\in\mathscr{R}_n^\times\subset\mathscr{R}^{\infty,\times}$. Thus we may take $\mathscr{R}^\infty$ as $\mathcal{R}$ in the above proposition and hence $g\in G(\widehat{\mathscr{R}^\infty_{\brF}})/\mathcal{G}(\widehat{\mathscr{R}^\infty})$ arises from an element $h'\in G(\mathscr{R}^\infty_{\brF})$. Since $\mathcal{G}$ is finite type over $\mathcal{O}_{\brF}$, the element $h'$ descends to an element $h\in G(\mathscr{R}_n)$ for some $n$. As $\mathscr{R}_n$ is a frame for the smooth $\mathbf{k}$-algebra $R_n$, the element $h$ connects the points $g_0$ and $g_1$ in $X(\{\mu\},b)$, hence $g_0\sim g_1$. It remains therefore to prove the Proposition.

\begin{proof}[Proof of Proposition \ref{prop5}] We first consider the case of $GL_n$. Since $\mathcal{R}$ is $p$-torsion free, we have $\mathcal{R}_{\brF}/\mathcal{R}\rightarrow\mathscr{R}_{\brF}/\mathscr{R}$ is a bijection (cf. \cite[Theorem 1]{BL}), hence the natural map $$GL_n(\mathcal{R}_{\brF})/GL_n(\mathcal{R})\rightarrow GL_n(\mathscr{R}_{\brF})/GL_n(\mathscr{R})$$ is an injection: If $A,B\in GL_n(\mathcal{R}_{\brF})$ and $C\in GL_n(\mathscr{R})$ is such that $A=BC$, we have $B^{-1}A=C\in GL_n(\mathcal{R}_{\brF})\cap GL_n(\mathscr{R})\subset Mat_{n,n}(\mathcal{R})$. We have $\det(C)\in\mathcal{R}_{\brF}^\times\cap\mathscr{R}^\times=\mathcal{R}^\times$ where the equality follows by our assumption on $\mathcal{R}$, hence $C\in GL_n(\mathcal{R})$.

To show surjectivity, let $g\in GL_n(\mathscr{R}_{\brF})$. Let $s\in\mathbb{N}$ such that $g,g^{-1}\in \frac{1}{p^s}\text{Mat}_{n,n}(\mathscr{R})$. Then for any $m>0$, there exists $h\in\frac{1}{p^s}\text{Mat}_{n,n}(\mathcal{R})$ such that $$g-h=\delta\in p^m\text{Mat}_{n,n}(\mathscr{R}).$$ For $m$ sufficiently large, we have $$g^{-1}h=1-g^{-1}\delta\in1+p\text{Mat}_{n,n}(\mathscr{R})\subset GL_n(\mathscr{R})$$ and hence $h\in GL_n(\mathcal{R}_{\brF})$ since $\det(h)\in \mathcal{R}_{\brF}\cap\mathscr{R}_{\brF}^\times=\mathcal{R}^\times_{\brF}$. We have $g=h(1-h^{-1}\delta)$, and hence for $m$ sufficiently large, the image of $h$ in $GL_n(\mathscr{R}_{\brF})/GL_n(\mathscr{R})$ is equal to $g$.

For the case of general $G$, we will need the following result of \cite[Theorem 2.1.5.5]{Br} 

\begin{lemma}\label{lemma5}
Let $Y$ be a flat $\mathcal{O}_{\brF}$-scheme. Let $\mathcal{F}$ denote the category of exact, faithful tensor functors from representations of $\mathcal{G}$ on finite free $\mathcal{O}_{\brF}$-modules to vector
bundles on $Y$.

If $P$ is a $\mathcal{G}$-bundle on $Y$, and $V$ is a representation of $\mathcal{G}$ on a finite free $\mathcal{O}_{\brF}$ module, write $F_P(V ) = \mathcal{G}\backslash (P \times V )$. Then $P\mapsto F_P$ is an equivalence between the
category of $\mathcal{G}$-bundles on $Y$, and the category $\mathcal{F}$.
\end{lemma}

From an element $g\in G(\mathscr{R})$, we can construct the pair $(P,\tau)$ consisting of the trivial  $\mathcal{G}$-torsor $P$ and a trivialization of $\tau$ of $P$ over $\mathscr{R}_{\brF}$. The set of isomorphism classes of such objects can be identified with $G(\mathscr{R}_{\brF})/\mathcal{G}(\mathscr{R})$. We will show that $(P,\tau)$ descends to a $\mathcal{G}$-torsor $X$ over $\mathcal{R}$  equipped with trivialization over $\mathcal{R}_{\brF}$. 

By Lemma \ref{lemma5}, the pair $(P,\tau)$ gives rise to an exact faithful tensor functor $F_P$ which associates to a representation of $\mathcal{G}$ on a finite free $\mathcal{O}_{\brF}$ module $V$, the vector bundle $F_P(V)=\mathcal{G}\backslash P\times V$ on $\mathscr{R}$ ,
 together with an isomorphism $$\tau_V:V\otimes_{\mathcal{O}_{\brF}}\mathscr{R}_{\brF}\cong F_P(V)\otimes_{\mathscr{R}}\mathscr{R}_{\brF}.$$
Since $P$ is trivial $F_P(V)$ is a trivial vector bundle, hence the pair $(F_P(V),\tau_V)$ corresponds to an element $g\in GL_n(\mathscr{R}_{\brF})/GL_n(\mathscr{R})$. By the case of $GL_n$ proved above, this corresponds to an element $h\in GL_n(\mathcal{R}_{\brF})/GL_n(\mathcal{R})$ and hence a pair $(F_P'(V),\tau_V')$ consisting of  the trivial vector bundle over $\mathcal{R}$ together with a trivialization $\tau_V'$ over $\mathcal{R}_{\brF}$. Then $F_P'(V)$ is an exact faithful tensor functor from the category of representations of $\mathcal{G}$ on finite free $\mathcal{O}_{\brF}$-modules to vector bundles on $\mathcal{R}$, hence corresponds to a $\mathcal{G}$ torsor $X$ over $\mathcal{R}$, and $\tau_{V}'$ gives a trivialization of $X$ over $\mathcal{R}_{\brF}$. Since $F_P'(V)$ is trivial for all $V$, we have $X$ is trivial, hence the pair $(X,\tau'_{V})$ corresponds to an element $h\in G(\mathcal{R}_{\brF})/G(\mathcal{R})$ which maps to $g$.
\end{proof}

\end{document}